\documentclass[11pt]{article}

\usepackage{array}
\usepackage{booktabs}
\usepackage{enumitem}
\newcolumntype{C}[1]{>{\centering\arraybackslash}m{#1}}
\usepackage{subfig}
\usepackage{graphicx}
\usepackage{amssymb}
\usepackage{mathtools}
\usepackage{amsmath}
\usepackage{amsthm}
\usepackage{latexsym}
\usepackage{amsfonts}
\usepackage{etoolbox}
\usepackage{multirow}
\usepackage{float}
\usepackage{amsmath,color}
\usepackage{url}
\usepackage[hidelinks]{hyperref}
\usepackage[T1]{fontenc}
\usepackage{doi}

\usepackage{arydshln}

\usepackage{lipsum}

\newtheorem{theorem}{Theorem}{}
\newtheorem{lemma}[theorem]{Lemma}
\newtheorem{proposition}[theorem]{Proposition}
\newtheorem{corollary}[theorem]{Corollary}
\newtheorem{remark}[theorem]{Remark}

\newtheorem{problem}[theorem]{Problem}

\newcommand{\Spec}{\operatorname{Spec}}
\newcommand{\mult}{\operatorname{mult}}
\newcommand{\mm}{\operatorname{mm}}
\newcommand{\te}{\operatorname{te}}
\newcommand{\tm}{\operatorname{tm}}
\newtheorem{definition}[theorem]{Definition}

\renewcommand{\thetheorem}{\textcolor{black}{\arabic{theorem}}}

\usepackage[utf8]{inputenc}
\usepackage{authblk}

\usepackage{tikz}
\usetikzlibrary{positioning}

\DeclareMathOperator{\rank}{rank}

\newcommand{\rev}[1]{\textcolor{black}{#1}}
\newcommand{\revblue}[1]{\textcolor{black}{#1}}

\title{\bf Metric dimension and maximum eigenvalue multiplicity of graph Laplacians}

\author{Monther R. Alfuraidan}
\author{Suliman Khan}
\author{Lorenzo Sauras-Altuzarra}
\affil{{\small\rev{Department of Mathematics \& Interdisciplinary Center of Smart Mobility and Logistics,\\
King Fahd University of Petroleum \& Minerals, Dhahran 31261, Saudi Arabia}\\
\texttt{monther@kfupm.edu.sa, suliman.khan.1@kfupm.edu.sa, lorenzo.altuzarra@kfupm.edu.sa}}}

\date{}

\begin{document}

\maketitle

\begin{abstract} We study equality between metric dimension and the maximum eigenvalue multiplicity of the Laplacian and signless Laplacian matrices. These quantities are not ordered in general. Using the standard decomposition associated with twin classes, we separate forced multiplicities from quotient-spectrum contributions. We give complete equality classifications for complete multipartite graphs and core--satellite graphs, and enumerate the multipartite equality graphs by order. For graphs formed by attaching at least two leaves to every vertex of a connected core, the metric dimension and the multiplicity of the eigenvalue \(1\) are known; we show that this multiplicity is maximal for both matrices, and uniquely maximal except for \(P_3\). We enumerate this core--leaf family up to isomorphism and obtain a uniform asymptotic formula by order and metric dimension, together with the quadratic exponential growth rate \(1/18\). An exact census of connected graphs of orders two through seven identifies the first separating examples and examples without twins. Finally, \rev{for every odd \(r\in\mathbb{Z}_{\geq 3}\)}, we study \(\Theta_{2,2,r}\) and determine both maximum multiplicities: the two equality properties alternate with \(r\) modulo four, producing infinite separating families in both directions. \end{abstract}

\bigskip


\medskip


\section{Introduction}

The metric dimension of a graph was introduced independently by
Slater~\cite{Slater1975} and Harary and Melter
\cite{HararyMelter1976}. A set of vertices is resolving if every
vertex is uniquely determined by its vector of distances to that set,
and the minimum cardinality of a resolving set is the metric
dimension, denoted by \(\dim(G)\). Metric dimension has been studied
extensively, both as a combinatorial parameter and through
applications in navigation, network discovery, chemistry, and
localization; see
\rev{\cite{Chartrand2000,Khuller1996}}.

Spectral graph theory relates graph structure to the eigenvalues of
matrices associated with a graph. In this paper, we consider the
Laplacian matrix \(L(G)\) and the signless Laplacian matrix \(Q(G)\);
standard references include
\rev{\cite{Brouwer2012,Cvetkovic2007,Cvetkovic2010,GodsilRoyle2001}}.
\rev{For each choice \(X\in\{L,Q\}\), define the \emph{maximum multiplicity} (abbreviated \(\mm_X\)) by}
\[
\rev{\mm_X}(G)=
\max_{\lambda\in\operatorname{Spec}(X(G))}
\operatorname{mult}_{X(G)}(\lambda).
\]
\rev{Here \(\operatorname{mult}_{X(G)}(\lambda)\) is the dimension of the eigenspace at \(\lambda\). We maximize the multiplicity over all eigenvalues, rather than taking the multiplicity of the largest eigenvalue.}
Our purpose is to compare \(\rev{\mm_X}(G)\) with \(\dim(G)\), and in
particular to determine when
\(
\rev{\mm_X}(G)=\dim(G).
\)

Neither \(\rev{\mm_L}(G)\) nor \(\rev{\mm_Q}(G)\) is universally bounded by \(\dim(G)\), and the reverse inequalities also fail. For example, the four-dimensional hypercube has Laplacian and signless Laplacian characteristic polynomial
\[
z(z-8)(z-6)^4(z-4)^6(z-2)^4,
\]
as follows from the Cartesian-product spectrum. Its vertices are binary strings of length four, and \(\{0000,1000,0100,0010\}\) is resolving: distance to \(0000\) gives the weight, while the other three distances determine the first three coordinates and hence the fourth. Thus \(\rev{\mm_L}=\rev{\mm_Q}=6>\dim(G)\), since \(\dim(G)\leq4\). In the opposite direction, Corollary~\ref{CoreSatelliteLQSeparation} gives \(\rev{\mm_L}(K_2\vee2K_2)=2<3=\dim(K_2\vee2K_2)\), and Theorem~\ref{ThetaAlternation} gives \(\rev{\mm_Q}(\Theta_{2,2,3})=1<2=\dim(\Theta_{2,2,3})\). The equality question therefore does not describe saturation of a universal inequality between the two parameters.

Farhan and Baskoro~\cite{FarhanBaskoro2026} studied metric dimension and adjacency nullity, as well as bounds involving maximum Laplacian and distance eigenvalue multiplicities for trees. Their subsequent work on rooted products~\cite[Proposition~3.3]{FarhanBaskoroFurther2026} includes the metric-dimension formula for the core--leaf construction considered here. For trees, as for all bipartite graphs, \(L\) and \(Q\) are similar; distinctions between the two equality properties therefore require nonbipartite graphs.

Twin vertices provide a natural link between metric dimension and
spectral multiplicity. Vertices in the same twin class cannot be
distinguished by vertices outside that class, yielding the standard
lower bound
\[
\dim(G)\geq\sum_C\bigl(|C|-1\bigr),
\]
\rev{where the sum is over the twin classes containing at least two vertices; see}
\cite{Hernando2010}. On the spectral side, false and true twin
classes generate repeated Laplacian and signless Laplacian
eigenvalues
\cite{Abreu2012,GroneMerris1994}.

The same symmetry gives invariant zero-sum subspaces for the two matrices. We define these spaces and the associated quotient matrices in Section~\ref{sec:twins}; their orthogonal decomposition separates eigenvalue multiplicities forced by twins from additional multiplicities in the quotient. Here a collision means that the same eigenvalue occurs in two of these restrictions, so their multiplicities add. Core--leaf graphs illustrate equality already accounted for by twin subspaces, whereas the separating theta and core--satellite families require such a collision. The twin-free examples show that these mechanisms do not exhaust the equality question.

The main results concern the equality classifications, rather than the standard decomposition itself. In Section~\ref{sec:twins}, we summarize the twin and quotient contributions and specialize known \(H\)-join factorizations~\cite{CardosoHJoin2013,SaravananHJoin2021}. Section~\ref{sec:multipartite} determines the maximum multiplicities for complete multipartite graphs, including all signless Laplacian collisions, and classifies and counts the equality cases. The two classifications differ for exactly five graphs.

Section~\ref{sec:satellite} treats \(K_c\vee(\eta K_s)\). Its adjacency and Laplacian spectra are known~\cite{EstradaBenziCoreSatellite}, \revblue{and its metric dimension for \(c=1\) follows from the known formula for vertex-amalgamations of complete graphs~\cite{SimanjuntakAmalgamation2015}.} We give a unified metric argument and classify both spectral equality properties. In particular, \(K_2\vee(\eta K_2)\) satisfies, \rev{for every \(\eta\in\mathbb{Z}_{\geq 2}\)}, only the signless Laplacian equality.

For the core--leaf construction in Section~\ref{sec:leaf}, both the metric dimension~\rev{\cite{FarhanBaskoroFurther2026,Iswadi2008}} and the exact multiplicity of the eigenvalue \(1\)~\cite{AndradeCardosoPastenRojo2015} are known. Our spectral argument bounds every other multiplicity, establishing simultaneous equality and uniqueness of the maximum except for \(P_3\). We then count these graphs up to isomorphism, obtain a uniform asymptotic estimate, and determine their quadratic exponential growth rate. \revblue{Proposition~\ref{CoreLeafOneSingleLeaf} also allows one core vertex of degree at least two to receive just one leaf, while preserving both equalities and uniqueness of the maximum.}

Section~\ref{sec:census} gives an exact census and its finite enumeration procedure for orders two through seven. Section~\ref{sec:theta} studies \(\Theta_{2,2,r}\) \rev{for every odd \(r\in\mathbb{Z}_{\geq 3}\)}. Its metric dimension follows from the classification of Knor, Sedlar, and \v{S}krekovski~\cite{Knor2022}, and the multiplicity of the signless Laplacian eigenvalue \(2\) is a specialization of Paone and Paone~\cite[Corollary~8.1]{PaonePaone2026}. We control the rest of the spectrum and the Laplacian spectrum to determine both maximum multiplicities. The resulting congruence conditions modulo four give infinite separating families in both directions. Section~\ref{sec:conclusion} records some remaining questions.

\section{Preliminaries}

All graphs are finite and simple, and \(G\) is connected unless stated otherwise. We use \(V(G)\) and \(E(G)\) for its vertex and edge sets, and \(|V(G)|\) for its order. For every vertex \(v\), its open neighborhood is \(N(v)=\{u:uv\in E(G)\}\), its closed neighborhood is \(N[v]=N(v)\cup\{v\}\), and its degree is \(\deg_G(v)=|N(v)|\). For every pair of vertices \(u,v\), \(d_G(u,v)\), or \(d(u,v)\) when the graph is clear, is the length of a shortest path between them.

\rev{For every \(n\in\mathbb{Z}_{\geq1}\)}, \(K_n\) and \(P_n\) denote the complete graph and the path on \(n\) vertices. \rev{For every \(k\in\mathbb{Z}_{\geq 2}\)} and every choice of \rev{\(n_1,\ldots,n_k\in\mathbb{Z}_{\geq1}\)}, the complete multipartite graph \(K_{n_1,\ldots,n_k}\) has independent parts of those sizes and every edge between distinct parts. The complement \(\overline G\) has the same vertices as \(G\) and exactly the missing edges between distinct vertices. The join \(G\vee H\) of two vertex-disjoint graphs retains their original edges and adds the edge \(\{u,v\}\) for every \(u\in V(G)\) and every \(v\in V(H)\). Thus each vertex of one graph is adjacent to every vertex of the other. \rev{For every \(a\in\mathbb{Z}_{\geq1}\)}, \(aH\) denotes the disjoint union of \(a\) copies of \(H\).

For a real symmetric matrix \(M\), \(\Spec(M)\) is its multiset of eigenvalues, and \(\mult_M(\lambda)=\dim\ker(M-\lambda I)\) is the multiplicity of a real number \(\lambda\). It is zero when \(\lambda\) is not an eigenvalue. The notation \(\lambda^{(a)}\) in a spectrum records \(a\) copies of \(\lambda\). An eigenvalue is \emph{simple} when its multiplicity is one, and a spectrum is simple when every eigenvalue is simple. The \emph{nullity} of a matrix is the dimension of its kernel. We write \(I_n\) for the identity matrix, \(M^T\) for the transpose, \(\operatorname{diag}(a_1,\ldots,a_n)\) for a diagonal matrix, and \(\chi_M(z)=\det(zI-M)\) for the characteristic polynomial; the order of an unsubscripted \(I\) is determined by \(M\).

We keep the conventional symbols \(A,D,L,Q\), \(\dim\), \(\Spec\), and \(\mult\). The paper-specific abbreviations \(\mm\), \(\te\), and \(\tm\) stand for maximum multiplicity, twin excess, and twin-induced multiplicity, respectively. Within each section, the graph families and parameters are defined before the results; subsequent statements retain these definitions unless an additional assumption is specified. For every \(a\in\mathbb{Z}\), write \(\mathbb{Z}_{\geq a}=\{m\in\mathbb{Z}:m\geq a\}\). Integer index ranges in sums and products are understood to be inclusive. When a sum is written as \(\sum(a_i:i\in I,\ P(i))\), the conditions after the colon specify its indexed terms; equal values at different indices are counted separately.

\subsection{Metric dimension}

Let
\(
S=\{s_1,\dots,s_t\}\subseteq V(G).
\)
The \emph{distance representation} of a vertex \(v\in V(G)\) with respect to \(S\) is defined by
\(
r(v|S)=
(d(v,s_1),\dots,d(v,s_t)).
\)
The set \(S\) is called a \emph{resolving set} if
\(
r(u|S)\neq r(v|S)
\)
for all distinct vertices \(u,v\in V(G)\).
The \emph{metric dimension} of \(G\), denoted by \(\dim(G)\), is the minimum cardinality of a resolving set.

\subsection{Laplacian and signless Laplacian matrices}

The adjacency matrix, degree matrix, Laplacian matrix, and signless Laplacian matrix of a graph \(G\) are denoted by \(A(G)\), \(D(G)\), \(L(G)\), and \(Q(G)\), respectively, where
\(L(G)=D(G)-A(G), \qquad Q(G)=D(G)+A(G).\)
The signless Laplacian matrix and its basic spectral properties are discussed in~\cite{Cvetkovic2007}.

Thus \(\mm_L(G)\) is the largest multiplicity attained by any eigenvalue of \(L(G)\), and \(\mm_Q(G)\) is defined analogously for \(Q(G)\). Neither notation denotes the multiplicity of the largest eigenvalue. Empty indexed families are omitted when taking a maximum.

\subsection{Twin vertices}

Twin structures play a central role throughout this paper.

\begin{definition}
Two distinct vertices \(u,v\in V(G)\) are called \emph{twins} if
\(
N(u)\setminus\{v\}=N(v)\setminus\{u\}.
\)
They are called \emph{false twins} if \(uv\notin E(G)\), equivalently
\(N(u)=N(v)\), and \emph{true twins} if \(uv\in E(G)\), equivalently
\(N[u]=N[v]\).

The twin relation, together with equality, is an equivalence relation
on \(V(G)\). Its equivalence classes are called \emph{twin classes}.
\rev{A twin class is \emph{nontrivial} when it contains at least two vertices. The vertices of each such class form either an independent set (a false twin class) or a clique (a true twin class). A graph is \emph{twin-free} when it has no nontrivial twin classes; the twin relation is a relation between vertices.}
\end{definition}

Let
\(
\mathcal{T}=\{C_1,\ldots,C_t\}
\)
denote the collection of all nontrivial twin classes of $G$. We define
the \emph{twin excess} of $G$ by
\[
\rev{\te}(G)=\sum_{i=1}^{t}\bigl(|C_i|-1\bigr).
\]
Thus, $\rev{\te}(G)$ gives the total contribution of the nontrivial twin
classes to the standard lower bound for the metric dimension.

The following standard bound is a direct consequence of the twin property; see Hernando et al.~\cite[Lemma~2.3 and Corollary~2.4]{Hernando2010}.

\begin{lemma}\label{TwinLowerBound}
\rev{The twin excess satisfies \(\dim(G)\geq\rev{\te}(G)\).}
\end{lemma}

\begin{proof}
No vertex outside a twin class \(C\) distinguishes two vertices of
\(C\). Thus every resolving set contains at least \(|C|-1\) vertices
from each nontrivial twin class. Summing over all nontrivial twin
classes gives
\[
\dim(G)\geq
\sum_{C\in\mathcal T}(|C|-1)
=\rev{\te}(G).
\]
\end{proof}

\section{Twin Subspaces and Quotient Spectra}\label{sec:twins}

We collect the standard spectral facts needed below. For every nontrivial twin class \(C\), its vertices have the same degree; write \(\deg(C)=\deg_G(v)\) for any \(v\in C\). For a vector \(x\in\mathbb R^{V(G)}\), its coordinate \(x_v\) is indexed by vertex \(v\), and its support is \(\operatorname{supp}(x)=\{v:x_v\ne0\}\). Put
\[
W_C=\left\{x\in\mathbb R^{V(G)}:\operatorname{supp}(x)\subseteq C,
\ \sum_{v\in C}x_v=0\right\},\qquad \dim W_C=|C|-1.
\]

Let \(\Pi\) be the partition of \(V(G)\) into all twin classes, including singletons. The parts of a partition will be called its \emph{cells}. Define
\[
\mathcal U=\{x\in\mathbb R^{V(G)}:
x_u=x_v\text{ for every }C\in\Pi\text{ and every }u,v\in C\}.
\]
The common value may differ between cells; a singleton imposes no constraint. A subspace \(W\) is \(M\)-invariant when \(Mx\in W\) for every \(x\in W\).

For a matrix \(M\) indexed by \(V(G)\), a partition \(\Pi=\{V_1,\ldots,V_q\}\) is \emph{equitable for \(M\)} if, for every pair of indices \(i,j\), the sum
\(\sum_{v\in V_j}M_{uv}\) has the same value for every \(u\in V_i\). Its \emph{row-sum quotient} \(R_{M,\Pi}\) has that value as entry \((i,j)\). For the adjacency matrix, this says that every vertex in \(V_i\) has the same number of neighbors in \(V_j\). The twin partition is equitable for \(A(G)\), \(L(G)\), and \(Q(G)\).

If \(M\) is symmetric, its restriction to cell-constant vectors is represented in the orthonormal basis \(\mathbf1_{V_i}/\sqrt{|V_i|}\) by the \emph{normalized quotient}
\[
\widehat R_{M,\Pi}=S_\Pi R_{M,\Pi}S_\Pi^{-1},\qquad
S_\Pi=\operatorname{diag}(\sqrt{|V_1|},\ldots,\sqrt{|V_q|}).
\]
Here \(\mathbf1_{V_i}\) equals one on \(V_i\) and zero elsewhere. Normalization refers to these basis vectors having unit norm. Two matrices \(M,N\) are \emph{similar} when \(N=S^{-1}MS\) for some invertible matrix \(S\); they are diagonally similar when \(S\) is diagonal. Similar matrices have the same characteristic polynomial.
 The twin eigenvalues and their equitable-partition interpretation are standard; see~\rev{\cite{Abreu2012,AbreuJustelMarkenzon2020,Barrett2017,Cardoso2007,GroneMerris1994,Kirkland2023}}.

\begin{proposition}\label{ExactMultiplicity}
For each choice \(X\in\{L,Q\}\), the orthogonal decomposition
\[
\mathbb R^{V(G)}=\mathcal U\oplus\bigoplus_{C\in\mathcal T}W_C
\]
is \(X(G)\)-invariant. For every \(C\in\mathcal T\) and every \(x\in W_C\), one has \(X(G)x=\theta_C^Xx\), where
\[
\begin{array}{c|cc}
 &\theta_C^L&\theta_C^Q\\\hline
C\text{ a false twin class}&\rev{\deg(C)}&\rev{\deg(C)}\\
C\text{ a true twin class}&\rev{\deg(C)}+1&\rev{\deg(C)}-1
\end{array}
\]
Writing \(B_X=\widehat R_{X(G),\Pi}\) for this normalized quotient, for every real \(\lambda\) we have
\[
\mult_{X(G)}(\lambda)=
\sum\bigl(|C|-1:C\in\mathcal T,\ \theta_C^X=\lambda\bigr)
+\mult_{B_X}(\lambda).
\]
\end{proposition}
\begin{proof}
For every \(x\in W_C\), coordinates outside \(C\) vanish after applying \(L\) or \(Q\), because each such vertex is adjacent to either all or none of \(C\). Inside \(C\), the adjacency contribution is zero for false twins and \(-x\) for true twins, giving the displayed eigenvalues. The spaces \(W_C\) are mutually orthogonal, and their orthogonal complement is exactly \(\mathcal U\). Since \(X(G)\) is symmetric, this complement is invariant. Taking eigenspace dimensions gives the multiplicity identity. This is the usual equitable-partition decomposition specialized to twins~\cite[Chapter~9]{GodsilRoyle2001}.
\end{proof}

We say that a class \(C\) \emph{induces} \(\theta_C^X\) precisely in the scalar-action sense just specified. Its contribution to that multiplicity is \(\dim W_C=|C|-1\); the quotient contributes the additional term \(\mult_{B_X}(\lambda)\). Define
\[
\rev{\tm_X}(G)=\max_{\lambda\in\mathbb R}
\sum\bigl(|C|-1:C\in\mathcal T,\ \theta_C^X=\lambda\bigr),
\]
with value zero when \(\mathcal T\) is empty. Then
\[
\rev{\tm_X}(G)\leq\rev{\te}(G)\leq\dim(G),\qquad \rev{\mm_X}(G)\geq\rev{\tm_X}(G).
\]
These are bounds from a common parameter, not a comparison between \(\rev{\mm_X}(G)\) and \(\dim(G)\).

\begin{remark}\label{MainTheoremUnified}
The multiplicity identity gives a sufficient criterion: if \(\rev{\tm_X}(G)=\dim(G)\), no twin-induced eigenvalue belongs to \(\Spec(B_X)\), and every eigenvalue of \(B_X\) has multiplicity at most \(\rev{\tm_X}(G)\), then \(\rev{\mm_X}(G)=\dim(G)\). The last condition follows, for example, if the number of cells of the twin partition is at most \(\rev{\tm_X}(G)\).

This criterion has a restricted scope. Its first hypothesis requires a tight twin bound and, when there are nontrivial twins, requires all their contributions to occur at a single eigenvalue. \rev{For every odd \(r\in\mathbb{Z}_{\geq 3}\)}, the graph \(\Theta_{2,2,r}\) has \(\rev{\tm_X}=\rev{\te}=1<2=\dim\); in \(K_2\vee(\eta K_2)\), \(\rev{\tm_Q}=\eta<\eta+1=\dim\). Their equality cases arise through quotient collisions and require the separate arguments below. For a core with two leaves attached at each of its \(h\) vertices, the quotient has order \(2h\), whereas \(\rev{\tm_X}=h\); bounding its dimension alone is insufficient. The proof of Theorem~\ref{CoreLeafEquality} instead controls its eigenspaces.
\end{remark}

\subsection{\rev{Mixed twin blow-ups}}\label{sec:blowups}

The \(H\)-join replaces every vertex \(i\) of a graph \(H\) by a graph \(F_i\) on a disjoint vertex set, and adds every edge between \(F_i\) and \(F_j\) exactly when \(ij\in E(H)\). These replacement vertex sets are its cells. We specialize the known \(H\)-join factorizations~\cite{CardosoHJoin2013,SaravananHJoin2021} to complete or edgeless cells.

\rev{Fix \(h\in\mathbb{Z}_{\geq1}\) and a connected graph \(H\) on \(\{1,\ldots,h\}\). For every \(i\in\{1,\ldots,h\}\), fix \(n_i\in\mathbb{Z}_{\geq1}\) and a type \(\varepsilon_i\in\{0,1\}\).} Replace vertex $i$
by a cell $V_i$ of size $n_i$, inducing an independent set when
$\varepsilon_i=0$ and a clique when $\varepsilon_i=1$. Join every vertex
of $V_i$ to every vertex of $V_j$ precisely when $ij\in E(H)$. Denote the
resulting graph by
\(
G=H[(n_1,\varepsilon_1),\ldots,(n_h,\varepsilon_h)].
\)
\rev{We require \(G\) to be connected; the only excluded case is \(h=1\), \(\varepsilon_1=0\), \(n_1>1\).}
\rev{For every \(i\in\{1,\ldots,h\}\), define the number \(\delta_i\) of neighbors outside \(V_i\) and the degree \(\deg_i\) of each vertex in \(V_i\) by}
\[
\delta_i=\sum_{j\sim_H i}n_j,
\qquad \rev{\deg_i}=\delta_i+\varepsilon_i(n_i-1).
\]

The following proposition is the corresponding specialization of the
known \(H\)-join factorization.

\begin{proposition}\label{BlowUpFactorization}
Define symmetric matrices $\widehat B_L$ and $\widehat B_Q$ of order $h$ by \rev{the following diagonal and off-diagonal formulas, where the latter apply only for \(i\ne j\):}
\[
(\widehat B_L)_{ii}=\delta_i,
\qquad
(\widehat B_L)_{ij}=
\begin{cases}-\sqrt{n_i n_j},&ij\in E(H),\\0,&ij\notin E(H),\end{cases}
\]
and
\[
(\widehat B_Q)_{ii}=\delta_i+2\varepsilon_i(n_i-1),
\qquad
(\widehat B_Q)_{ij}=
\begin{cases}\sqrt{n_i n_j},&ij\in E(H),\\0,&ij\notin E(H).\end{cases}
\]
Then
\begin{align*}
\det(zI-L(G))
 &=\det(zI-\widehat B_L)
   \prod_{i=1}^h(z-\rev{\deg_i}-\varepsilon_i)^{n_i-1},\\
\det(zI-Q(G))
 &=\det(zI-\widehat B_Q)
   \prod_{i=1}^h(z-\rev{\deg_i}+\varepsilon_i)^{n_i-1}.
\end{align*}
Consequently, for $X\in\{L,Q\}$,
\[
\rev{\mm_X}(G)=\max_{\lambda\in\mathbb R}
\left\{\mult_{\widehat B_X}(\lambda)
+\sum_{\theta_i^X=\lambda}(n_i-1)\right\},
\]
where $\theta_i^L=\rev{\deg_i}+\varepsilon_i$ and
$\theta_i^Q=\rev{\deg_i}-\varepsilon_i$.
\end{proposition}

\begin{proof}
For each $i$, let
\[
\rev{W_i=\left\{x:\operatorname{supp}(x)\subseteq V_i,
\ \sum_{v\in V_i}x_v=0\right\}.}
\]
On $W_i$, the matrices $L(G)$ and $Q(G)$ act as multiplication by
$\rev{\deg_i}+\varepsilon_i$ and $\rev{\deg_i}-\varepsilon_i$, respectively. Moreover,
\[
\mathbb R^{V(G)}=\mathcal U\oplus\bigoplus_{i=1}^hW_i,
\]
where $\mathcal U$ consists of vectors constant on every cell. In the
orthonormal basis $e_i=\mathbf1_{V_i}/\sqrt{n_i}$ of $\mathcal U$, direct
calculation gives $\widehat B_L$ and $\widehat B_Q$ as the matrices of the
two restrictions. The factorizations and multiplicity formula follow.
\end{proof}

\begin{remark}

For this cell partition, write \(R_X\) for the row-sum quotient of \(X(G)\), for each choice \(X\in\{L,Q\}\). With \(S=\operatorname{diag}(\sqrt{n_1},\ldots,\sqrt{n_h})\), one has \(\widehat B_X=SR_XS^{-1}\). Thus the two explicitly defined quotient matrices are diagonally similar. The normalization used in Proposition~\ref{BlowUpFactorization} is the unit-vector normalization defined above, not a normalization of the graph Laplacian.

\end{remark}

\section{Complete Multipartite Graphs}\label{sec:multipartite}

Laplacian spectra of complete multipartite graphs are classical; see, for example, Brouwer and Haemers~\cite{Brouwer2012}. Using the quotient factorization developed in \rev{Section~\ref{sec:blowups}}, we treat the signless Laplacian simultaneously and determine the maximum Laplacian and signless Laplacian eigenvalue multiplicities. We then obtain a complete classification of the complete multipartite graphs satisfying
\(\dim(G)=\rev{\mm_L}(G),\)
and separately those satisfying
\(\dim(G)=\rev{\mm_Q}(G).\)

\rev{\rev{Fix \(k\in\mathbb{Z}_{\geq2}\)} and \rev{\(n_1,\ldots,n_k\in\mathbb{Z}_{\geq1}\)}. Throughout this section let \(G=K_{n_1,\ldots,n_k}\), and put}
\(
n=\sum_{i=1}^k n_i.
\)
Let
\(
\mathcal S=\{n_i:1\le i\le k\}
\)
be the set of distinct part sizes. For every \(s\in\mathcal S\), put
\(
a_s=|\{i:n_i=s\}|.
\)
Thus,
\[
\sum_{s\in\mathcal S}a_s=k,
\qquad
\sum_{s\in\mathcal S}a_ss=n.
\]
We also put
\(
r=a_1
\)
when \(1\in\mathcal S\), and \(r=0\) otherwise.

\revblue{The metric-dimension formula for complete multipartite graphs, including singleton parts, is recorded by Okamoto, Phinezy, and Zhang~\cite[Lemma~2.7]{Okamoto2010}.} 
The next proposition records the metric-dimension formula, including singleton parts, and separately records the classical Laplacian spectrum and its maximum multiplicity. The metric argument uses the twin-class bound~\cite{Hernando2010}; the spectral notation \(\lambda^{(a)}\) means that \(\lambda\) occurs with multiplicity \(a\).

\begin{proposition}\label{MultipartiteMetricSpectrum}
\rev{With the notation fixed above,}
\(
\dim(G)=n-k+\max\{r-1,0\},
\)
and
\(
\Spec(L(G))=
\{0^{(1)},n^{(k-1)}\}
\cup\rev{\bigcup_{s\in\mathcal S\setminus\{1\}}}
\{(n-s)^{(a_s(s-1))}\}.
\)
Consequently,
\[
\rev{\mm_L}(G)=
\max\left\{
k-1,
\rev{\max\{a_s(s-1):s\in\mathcal S,\ s\ge2\}}
\right\}.
\]
\end{proposition}

\begin{proof}
Every nonsingleton part is a false twin class, while the vertices in all
singleton parts form one true twin class when \(r\ge2\). Hence every
resolving set has size at least
\[
\rev{\sum\bigl(a_s(s-1):s\in\mathcal S,\ s\ge2\bigr)}
+\max\{r-1,0\}
=
n-k+\max\{r-1,0\}.
\]
Selecting all but one vertex from every nonsingleton part and, if
\(r\ge2\), all but one singleton vertex gives a resolving set of this
size.

View \(G\) as an independent-set blow-up of \(K_k\). A part of size
\(n_i\) induces the Laplacian eigenvalue \(n-n_i\) with contribution
\(n_i-1\). The normalized Laplacian quotient is \(nI_k-uu^T\), where
\(u=(\sqrt{n_1},\ldots,\sqrt{n_k})^T.\)
Since \(u^Tu=n\), this quotient has spectrum
\(
\{0^{(1)},n^{(k-1)}\}.
\)
The assertions concerning the Laplacian spectrum and its maximum
eigenvalue multiplicity follow.
\end{proof}

Using Proposition~\ref{BlowUpFactorization}, we group parts of equal
size to obtain the following distinct-part-size reduction. Its purpose
is to make the repeated
quotient eigenvalues, the remaining simple roots, and all possible
\rev{eigenvalue collisions explicit: a collision occurs when an eigenvalue supplied by one factor of the decomposition is also supplied by another, so the corresponding multiplicities add.} This form will be used to determine
the maximum signless Laplacian eigenvalue multiplicity and to obtain
the corresponding equality classification.

\begin{theorem}\label{QQuotientReduction}
\rev{The signless Laplacian characteristic polynomial is}
\[
\det(zI-Q(G))
=
\prod_{s\in\mathcal S}
\bigl(z-(n-s)\bigr)^{a_s(s-1)}
\bigl(z-(n-2s)\bigr)^{a_s-1}
P_G(z),
\]
where
\[
P_G(z)
=
\prod_{s\in\mathcal S}\bigl(z-(n-2s)\bigr)
-
\sum_{s\in\mathcal S}
a_ss
\rev{\prod_{t\in\mathcal S\setminus\{s\}}}
\bigl(z-(n-2t)\bigr).
\]
Moreover, \(P_G\) has degree \(|\mathcal S|\), all its roots are simple,
\rev{and \(P_G(n-2s)\ne0\) for every \(s\in\mathcal S\).}
\end{theorem}

\begin{proof}
By Proposition~\ref{BlowUpFactorization},
\[
\det(zI-Q(G))
=
\det(zI_k-S_Q)
\prod_{i=1}^k\bigl(z-(n-n_i)\bigr)^{n_i-1},
\]
where
\(
S_Q
=
\operatorname{diag}(n-2n_1,\ldots,n-2n_k)+uu^T
\)
and
\(u=(\sqrt{n_1},\ldots,\sqrt{n_k})^T.\)
Grouping the factors corresponding to equal part sizes gives
\[
\prod_{i=1}^k\bigl(z-(n-n_i)\bigr)^{n_i-1}
=
\prod_{s\in\mathcal S}
\bigl(z-(n-s)\bigr)^{a_s(s-1)}.
\]

For each \(s\in\mathcal S\), consider the coordinates of \(S_Q\)
corresponding to the \(a_s\) parts of size \(s\). The subspace of
vectors supported on these coordinates whose coordinate sum is zero
has dimension \(a_s-1\). On this subspace, \(S_Q\) acts as scalar
multiplication by \(n-2s\). Hence \(n-2s\) is an eigenvalue of \(S_Q\)
with multiplicity at least \(a_s-1\).

The orthogonal complement of the direct sum of these zero-sum
subspaces is spanned by the normalized vectors that are constant on
the coordinates associated with each distinct part size. In this
basis, the remaining restriction is
\(
R=
\operatorname{diag}(n-2s:s\in\mathcal S)+vv^T,
\)
where
\(v=(\sqrt{a_ss}:s\in\mathcal S)^T.\)
Consequently,
\[
\det(zI_k-S_Q)
=
\prod_{s\in\mathcal S}
\bigl(z-(n-2s)\bigr)^{a_s-1}
\det(zI-R).
\]
By the matrix determinant lemma,
\[
\begin{aligned}
\det(zI-R)
&=
\prod_{s\in\mathcal S}\bigl(z-(n-2s)\bigr)
\left(
1-\sum_{s\in\mathcal S}
\frac{a_ss}{z-(n-2s)}
\right)\\
&=
\prod_{s\in\mathcal S}\bigl(z-(n-2s)\bigr)\\
&\quad-
\sum_{s\in\mathcal S}
a_ss
\rev{\prod_{t\in\mathcal S\setminus\{s\}}}
\bigl(z-(n-2t)\bigr)\\
&=P_G(z).
\end{aligned}
\]

The diagonal entries \(n-2s\), \(s\in\mathcal S\), are distinct, and
every coordinate of \(v\) is nonzero. Therefore, the eigenvalues of
the rank-one perturbation
\[
R=\operatorname{diag}(n-2s:s\in\mathcal S)+vv^T
\]
strictly interlace the diagonal entries. It follows that all roots of
\(P_G\) are simple and none is equal to \(n-2s\). Combining the
factorizations proves the result.
\end{proof}

\rev{For every \(t\in\mathcal S\) with \(t\ge2\), define}
\[
\varepsilon_t=
\begin{cases}
1,
&
\displaystyle
t/2\notin\mathcal S
\quad\text{and}\quad
1+\sum_{s\in\mathcal S}\frac{a_ss}{t-2s}=0,\\[3mm]
0,&\text{otherwise}.
\end{cases}
\]

\begin{corollary}\label{ExactQMultiplicityMultipartite}
For every \(t\in\mathcal S\) with \(t\ge2\),
\[
\mult_{Q(G)}(n-t)
=
a_t(t-1)+
\begin{cases}
a_{t/2}-1,
& t \text{ is even and }t/2\in\mathcal S,\\[1mm]
\varepsilon_t,
& t/2\notin\mathcal S.
\end{cases}
\]

Moreover, \rev{using the multiplicities just displayed,}

\[
\begin{aligned}
\mm_Q(G)=\max\Bigl(&\{1\}\cup\{a_s-1:s\in\mathcal S,\ 2s\notin\mathcal S\}\\
 &{}\cup\{\mult_{Q(G)}(n-t):t\in\mathcal S,\ t\geq2\}\Bigr).
\end{aligned}
\]

\end{corollary}

\begin{proof}
The factor
\(\bigl(z-(n-t)\bigr)^{a_t(t-1)}\)
gives the twin-induced contribution to the multiplicity of \(n-t\).

This eigenvalue coincides with a repeated quotient eigenvalue
\(n-2s\) precisely when \(t=2s\). Thus, if \(t\) is even and
\(t/2\in\mathcal S\), the additional contribution is
\(a_{t/2}-1\).

If \(t/2\notin\mathcal S\), then \(n-t\) is not one of the values
\(n-2s\). By Theorem~\ref{QQuotientReduction}, it is a root of
\(P_G\) precisely when
\[
1-\sum_{s\in\mathcal S}
\frac{a_ss}{(n-t)-(n-2s)}=0,
\]
or equivalently,
\[
1+\sum_{s\in\mathcal S}\frac{a_ss}{t-2s}=0.
\]
Since every root of \(P_G\) is simple, this contributes exactly
\(\varepsilon_t\).

Finally, a repeated quotient eigenvalue \(n-2s\) that does not coincide
with a twin-induced eigenvalue has multiplicity \(a_s-1\), and all
remaining roots of \(P_G\) are simple. Taking the maximum over all
possibilities proves the formula for \(\rev{\mm_Q}(G)\).
\end{proof}

\begin{theorem}\label{MultipartiteClassification}
\rev{With the notation of this section,}
\( \rev{\mm_L}(G)=\dim(G) \)
if and only if at least one of the following conditions holds:
\begin{enumerate}[label=\textup{(\roman*)}]
\item \(r\leq1\), and all parts of size at least \(2\) have the same
size;
\item \(r\geq1\), and every part has size \(1\) or \(2\).
\end{enumerate}
\end{theorem}

\begin{proof}
Put
\[
D=\dim(G)=\sum_{s\ge2}a_s(s-1)+\max\{r-1,0\}.
\]
Both $k-1$ and every $a_s(s-1)$ are at most $D$. By
Proposition~\ref{MultipartiteMetricSpectrum}, equality holds precisely when
\[
k-1=D
\quad\text{or}\quad a_s(s-1)=D\text{ for some }s\ge2.
\]
The second alternative is equivalent to condition~\textup{(i)}. If
$r=0$, then
\[
D-(k-1)=1+\sum_{i=1}^k(n_i-2)>0.
\]
If $r\ge1$, then
\[
D-(k-1)=\sum_{n_i\ge2}(n_i-2),
\]
which vanishes precisely under condition~\textup{(ii)}.
\end{proof}

\begin{remark}
In condition~\textup{(i)}, there is at least one nonsingleton part because \(k\geq2\) and \(r\leq1\). If there is exactly one, the equal-size requirement is automatic. Condition~\textup{(ii)} imposes no minimum on the number of nonsingleton parts: there may be zero, one, or several. Zero nonsingleton parts gives a complete graph.
\end{remark}

\begin{theorem}\label{MultipartiteQClassification}
\rev{With the same notation,}
\(
\rev{\mm_Q}(G)=\dim(G)
\)
\rev{if and only if at least one of the following conditions holds:}
\begin{enumerate}[label=\textup{(\roman*)}]
\item \(r=0\) and all parts have one common size;

\item \(r=1\), all nonsingleton parts have one common size, and
\(
G\ne K_{1,3,3};
\)

\item \rev{\(r\geq2\), and every nonsingleton part has size \(2\) (this includes complete graphs, for which there are no nonsingleton parts);}

\item \(G\) is one of the four exceptional graphs
\(
K_{1,1,3,3,3},\
K_{1,1,4,4},\
K_{2,5,5,5},\
K_{2,6,6}.
\)
\end{enumerate}
\end{theorem}

\begin{proof}
\rev{For every \(s\in\mathcal S\) with \(s\ge2\), put}
\(
A_s=a_s(s-1),
\
R=\max\{r-1,0\}.
\)
By Proposition~\ref{MultipartiteMetricSpectrum},
\[
D:=\dim(G)=\rev{\sum\bigl(A_s:s\in\mathcal S,\ s\ge2\bigr)+R}.
\]

We use Corollary~\ref{ExactQMultiplicityMultipartite}. First consider a
repeated quotient eigenvalue \(n-2s\) that does not coincide with a
twin-induced eigenvalue. Its multiplicity is \(a_s-1\). If \(s\ge2\),
then
\(a_s-1<a_s(s-1)=A_s\le D.\)
If \(s=1\), its multiplicity is \(r-1\), which equals \(D\) precisely
when there are no nonsingleton parts. This gives the complete graphs,
which are contained in family~\textup{(iii)}. The remaining roots of
\(P_G\) are simple. They can therefore give the maximum multiplicity
only when \(D=1\), and the corresponding graphs are already contained
in the listed families.

It remains to consider a twin-induced eigenvalue \(n-t\), where
\(t\in\mathcal S\) and \(t\ge2\). Put
\[
B_t=D-A_t
=
\rev{\sum\bigl(A_s:s\in\mathcal S,\ s\ge2,\ s\ne t\bigr)+R}.
\]
Thus, \(B_t\) is the contribution to the metric dimension from all
twin classes other than the parts of size \(t\).

Suppose first that \(t\) is even and \(t/2\in\mathcal S\). By
Corollary~\ref{ExactQMultiplicityMultipartite},
\(\mult_{Q(G)}(n-t)=A_t+a_{t/2}-1.\)
If \(t/2\ge2\), then
\[
a_{t/2}-1
<
a_{t/2}\left(\frac{t}{2}-1\right)
=A_{t/2}
\le B_t.
\]
Hence the multiplicity is strictly smaller than \(D\).

The only remaining possibility is \(t/2=1\), or equivalently \(t=2\).
In this case,
\(\mult_{Q(G)}(n-2)=A_2+r-1.\)
This equals \(D\) precisely when no nonsingleton part has size other
than \(2\). When \(r=1\), these graphs are contained in
family~\textup{(ii)}, and when \(r\ge2\), they give
family~\textup{(iii)}.

We may now suppose that \(t/2\notin\mathcal S\). In this case,
\[
\mult_{Q(G)}(n-t)=A_t+\varepsilon_t,
\qquad
\varepsilon_t\in\{0,1\}.
\]
Therefore,
\(
\mult_{Q(G)}(n-t)=D
\)
if and only if
\(
B_t=\varepsilon_t.
\)
Since \(B_t\) is a nonnegative integer and
\(\varepsilon_t\in\{0,1\}\), we need only consider the cases
\(B_t=0\) and \(B_t=1\).

Suppose first that \(B_t=0\). Then \(r\le1\), and \(t\) is the only
nonsingleton part size.

If \(r=0\), then all parts have size \(t\). Since \(k=a_t\ge2\),
\[
1+\sum_{s\in\mathcal S}\frac{a_ss}{t-2s}
=
1-a_t\ne0.
\]
Thus \(\varepsilon_t=0\), giving family~\textup{(i)}.

If \(r=1\), then \(t\ge3\), because the case \(t=2\) was considered
above. The collision equation becomes
\(1+\frac{1}{t-2}-a_t=0.\)
This equation holds precisely when
\(
t=3,\ a_t=2.
\)
Thus \(\varepsilon_t=1\) only for \(K_{1,3,3}\). For that graph,
\(\rev{\mm_Q}(K_{1,3,3})=5>4=\dim(K_{1,3,3}).\)
All other graphs with \(r=1\) and one nonsingleton part size satisfy
the desired equality, giving family~\textup{(ii)}.

It remains to consider \(B_t=1\). There are exactly two possibilities.

First, \(R=1\) and there are no other nonsingleton part sizes. Thus
\(r=2\), and the collision equation is
\(1+\frac{2}{t-2}-a_t=0.\)
Equivalently,
\(a_t=1+\frac{2}{t-2}.\)
Since \(a_t\) is a positive integer and \(t\ge3\), this occurs
precisely when
\[
(t,a_t)=(3,3)
\quad\text{or}\quad
(t,a_t)=(4,2).
\]
These give
\(
K_{1,1,3,3,3}
\ \text{and}\
K_{1,1,4,4}.
\)

Second, \(R=0\), and the contribution from all other nonsingleton part
sizes is \(1\). Since
\(
A_s=a_s(s-1),
\)
this is possible only when there is exactly one part of size \(2\).
Thus the part sizes different from \(t\) consist of one part of size
\(2\), together with at most one singleton part.

Let \(r\in\{0,1\}\). The collision equation becomes
\(1+\frac{r}{t-2}+\frac{2}{t-4}-a_t=0.\)
If \(r=0\), then
\(a_t=1+\frac{2}{t-4}.\)
The positive integral solutions are
\[
(t,a_t)=(5,3)
\quad\text{and}\quad
(t,a_t)=(6,2),
\]
which give
\(
K_{2,5,5,5}
\ \text{and}\
K_{2,6,6}.
\)
If \(r=1\), the equation
\(a_t=1+\frac{1}{t-2}+\frac{2}{t-4}\)
has no positive integral solution for \(t\ge3\), \(t\ne4\).

These arguments show that the listed families are the only possible
equality cases. Conversely, substituting their part sizes into the
multiplicity formula of
Corollary~\ref{ExactQMultiplicityMultipartite} shows that every listed
graph satisfies
\(
\rev{\mm_Q}(G)=\dim(G).
\)
\end{proof}

\begin{corollary}\label{BalancedMultipartiteCor}
\rev{For every \(k,s\in\mathbb{Z}_{\geq2}\), the graph \(G=K_{s,\ldots,s}\) with \(k\) parts satisfies}
\(
\rev{\mm_L}(G)=\rev{\mm_Q}(G)=\dim(G)=k(s-1).
\)
\end{corollary}

\begin{proof}
The Laplacian assertion follows from
Theorem~\ref{MultipartiteClassification}, while the signless Laplacian
assertion follows from Theorem~\ref{MultipartiteQClassification}.
Alternatively, the signless Laplacian spectrum can be written explicitly as
\[
\Spec(Q(G))=
\{(n-s)^{(k(s-1))},(n-2s)^{(k-1)},(2n-2s)^{(1)}\}.
\]
The three eigenvalues are distinct and \(k(s-1)\ge k-1\).
\end{proof}

As boundary cases, the complete graph and \rev{star graph} satisfy
\[
\rev{\mm_L}(K_n)=\rev{\mm_Q}(K_n)=\dim(K_n)=n-1\quad\rev{\text{for every }n\in\mathbb{Z}_{\geq2}}
\]
and
\[
\rev{\mm_L}(K_{1,n-1})=\rev{\mm_Q}(K_{1,n-1})
=\dim(K_{1,n-1})=n-2\quad\rev{\text{for every }n\in\mathbb{Z}_{\geq3}}.
\]

\subsection{Enumeration of equality complete multipartite graphs}

The classifications obtained above also yield explicit enumeration formulae. Let \(\rev{\sigma_0(n)}\) denote the number of positive divisors of \(n\). \rev{\rev{For every \(n\in\mathbb{Z}_{\geq 2}\)}, let} \(a_L(n)\) and \(a_Q(n)\) denote, respectively, the numbers of nonisomorphic complete multipartite graphs \(G\) of order \(n\) satisfying \(\rev{\mm_L}(G)=\dim(G)\) and \(\rev{\mm_Q}(G)=\dim(G).\)

\begin{theorem}\label{MultipartiteEnumeration} \rev{\rev{For every \(n\in\mathbb{Z}_{\geq 2}\)},} $$ a_L(n) = \rev{\sigma_0(n)}+\rev{\sigma_0(n-1)}+\left\lfloor\frac n2\right\rfloor-3. $$ Moreover, $$ a_Q(n) = \begin{cases} a_L(n)-1,&n=7,\\ a_L(n)+1,&n\in\{10,11,14,17\},\\ a_L(n),&\text{otherwise}. \end{cases} $$ \end{theorem}

\begin{proof}

By Theorem~\ref{MultipartiteClassification}, a complete multipartite graph \(G\) satisfies \(\rev{\mm_L}(G)=\dim(G)\) if and only if it belongs to one of the following two families: \begin{enumerate}[label=\textup{(\roman*)}] \item there is at most one singleton part, and all nonsingleton parts have the same size; \item there is at least one singleton part, and every part has size \(1\) or \(2\). \end{enumerate}

We count the graphs in these families up to isomorphism.

First suppose that there are no singleton parts. Then $$ G=K_{\underbrace{s,\ldots,s}_{k}}, \qquad\rev{s,k\in\mathbb{Z}_{\geq2}}, $$ and \(n=ks.\)

Thus \(s\) may be any proper divisor of \(n\) different from \(1\), giving \(\rev{\sigma_0(n)}-2\) graphs.

Next suppose that there is exactly one singleton part and that all nonsingleton parts have a common size \(s\geq2\). Then $$ G=K_{1,\underbrace{s,\ldots,s}_{k}}, \qquad\rev{k\in\mathbb{Z}_{\geq1}}, $$ and \(n-1=ks.\)

Hence \(s\) may be any divisor of \(n-1\) different from \(1\), giving \(\rev{\sigma_0(n-1)}-1\) graphs.

Finally, consider the graphs whose parts all have size \(1\) or \(2\), with at least one singleton part. They have the form $$ G= K_{\underbrace{1,\ldots,1}_{r}, \underbrace{2,\ldots,2}_{q}}, \qquad\rev{r\in\mathbb{Z}_{\geq1},\ q\in\mathbb{Z}_{\geq0}}, $$ where \(n=r+2q.\)

Thus $$ 0\leq q\leq \left\lfloor\frac{n-1}{2}\right\rfloor, $$ so this family contains \(\left\lfloor\frac{n+1}{2}\right\rfloor\) graphs.

The second and third families overlap precisely when there is exactly one singleton part and every nonsingleton part has size \(2\). Such a graph exists precisely when \(n\) is odd. Therefore $$ \begin{aligned} a_L(n) &= \rev{\sigma_0(n)}-2+\rev{\sigma_0(n-1)}-1+ \left\lfloor\frac{n+1}{2}\right\rfloor - \begin{cases} 1,&n\text{ odd},\\ 0,&n\text{ even}. \end{cases}\\ &= \rev{\sigma_0(n)}+\rev{\sigma_0(n-1)}+ \left\lfloor\frac n2\right\rfloor-3. \end{aligned} $$

We now consider the signless Laplacian. By Theorem~\ref{MultipartiteQClassification}, its classification agrees with the preceding Laplacian classification except for five graphs. The graph \(K_{1,3,3}\) satisfies \(\rev{\mm_L}(G)=\dim(G)\) but not \(\rev{\mm_Q}(G)=\dim(G)\), whereas $$ K_{1,1,4,4},\qquad K_{1,1,3,3,3},\qquad K_{2,6,6},\qquad K_{2,5,5,5} $$ satisfy \(\rev{\mm_Q}(G)=\dim(G)\) but not \(\rev{\mm_L}(G)=\dim(G)\).
Their orders are, respectively, \(7,10,11,14,\) and \(17\). The stated formula for \(a_Q(n)\) follows.

\end{proof}

\begin{corollary}\label{MultipartiteEnumerationComparison} \rev{\rev{For every \(n\in\mathbb{Z}_{\geq 2}\)} with \(n\notin\{7,10,11,14,17\}\),} we have \(a_L(n)=a_Q(n).\) Moreover, among all complete multipartite graphs, the properties $$ \rev{\mm_L}(G)=\dim(G) \qquad\text{and}\qquad \rev{\mm_Q}(G)=\dim(G) $$ differ for exactly five graphs, namely $$ K_{1,3,3},\quad K_{1,1,4,4},\quad K_{1,1,3,3,3},\quad K_{2,6,6},\quad K_{2,5,5,5}. $$ \end{corollary}

The preceding enumeration can also be expressed in terms of ordinary
generating functions. Put $$ D(x)=\sum_{n\geq1}\rev{\sigma_0(n)}x^n =\sum_{k\geq1}\frac{x^k}{1-x^k}. $$ Then $$ \begin{aligned} A_L(x) &:=\sum_{n\geq2}a_L(n)x^n\\ &=(1+x)D(x)-x +\frac{x^2(1+x)}{(1-x^2)^2} -\frac{3x^2}{1-x}. \end{aligned} $$

Consequently, $$ \begin{aligned} A_Q(x) &:=\sum_{n\geq2}a_Q(n)x^n\\ &= A_L(x)-x^7+x^{10}+x^{11}+x^{14}+x^{17}. \end{aligned} $$

For reference, the first values are $$ \begin{array}{c|rrrrrrrrrrrr} n &2&3&4&5&6&7&8&9&10&11&12&13\\ \hline a_L(n) &1&2&4&4&6&6&7&8&9&8&11&11\\ a_Q(n) &1&2&4&4&6&5&7&8&10&9&11&11 \end{array} $$

\section{Core--Satellite Graphs}\label{sec:satellite}

We next consider a further class of graphs to which the twin-decomposition framework applies naturally. Following Estrada and Benzi \cite{EstradaBenziCoreSatellite}, \rev{consider \(K_c\vee(\eta K_s)\),} where \(\eta K_s\) denotes the disjoint union of \(\eta\) copies of \(K_s\). These graphs are known as core--satellite graphs. The same family also appears in the literature as Type I generalized windmill graphs \cite{KooijWindmill}.

Their adjacency and Laplacian spectral properties have been studied previously; in particular, the Laplacian spectrum was determined by Estrada and Benzi \cite{EstradaBenziCoreSatellite}. Our purpose here is different: we determine when the maximum Laplacian or signless Laplacian eigenvalue multiplicity equals the metric dimension.

\revblue{When \(c=1\), the graph \(K_1\vee(\eta K_s)\) is obtained from \(\eta\) copies of \(K_{s+1}\) by identifying one vertex from each copy as a single common vertex. This is their vertex-amalgamation. Since \(s\geq2\), the known formula for vertex-amalgamations of complete graphs, recorded by Simanjuntak, Uttunggadewa, and Saputro~\cite{SimanjuntakAmalgamation2015}, gives metric dimension \(\eta\bigl(\dim(K_{s+1})-1\bigr)=\eta(s-1)\). The argument below includes this known case and handles arbitrary \(c\).}

\rev{Throughout this section, \(c\in\mathbb{Z}_{\geq1}\) and \(s,\eta\in\mathbb{Z}_{\geq2}\), and \(G=K_c\vee(\eta K_s)\).}

The excluded cases are already covered by earlier sections: if \(\eta=1\), then \(\rev{K_c\vee K_s}=K_{c+s}\), while if \(s=1\), then \(\rev{K_c\vee(\eta K_1)}=K_c\vee\overline{K_\eta}\) is a complete multipartite graph.

\begin{theorem}\label{CoreSatelliteEquality} \rev{The graph defined above satisfies} \(\dim(G)=(c-1)+\eta(s-1).\) Moreover, \(\rev{\mm_L}(G)=\dim(G)\) if and only if \(c=1,\) whereas \(\rev{\mm_Q}(G)=\dim(G)\) if and only if \(c=1 \qquad\text{or}\qquad (c,s)=(2,2).\) \end{theorem}

\begin{proof}

Let \(C\) be the vertex set of the core \(K_c\), and let \(V_1,\ldots,V_\eta\) be the vertex sets of the \(\eta\) satellite copies of \(K_s\).

Each \(V_i\) is a true twin class of size \(s\). Moreover, when \(c\geq2\), the core \(C\) is a true twin class of size \(c\). Hence the twin lower bound gives \(\dim(G)\geq(c-1)+\eta(s-1).\) Conversely, choose all but one vertex from each \(V_i\), and, when \(c\geq2\), all but one vertex from \(C\). Let \(S\) be the resulting set. Then \(|S|=(c-1)+\eta(s-1).\)

Every vertex of \(S\) is distinguished from every other vertex by its own zero coordinate. Let \(u_i\) denote the unselected vertex of \(V_i\).
\rev{For every pair of distinct satellite indices \(i,j\), any} selected vertex of \(V_i\) has distance \(1\) from \(u_i\) and distance \(2\) from \(u_j\), and therefore distinguishes them.
Finally, the unselected core vertex, when it exists, is distinguished from \(u_i\) by any selected vertex in \(V_j\), \(j\ne i\): its distances to
these two vertices are \(1\) and \(2\), respectively. The same argument distinguishes the unique core vertex from every \(u_i\) when \(c=1\).
Thus \(S\) resolves \(G\), and consequently \(\dim(G)=D:=(c-1)+\eta(s-1).\)

We first consider the Laplacian matrix. The Laplacian spectrum of core--satellite graphs is known \cite{EstradaBenziCoreSatellite} and, in
our notation, is $$ \Spec(L(G)) = \left\{0^{(1)}, c^{(\eta-1)}, (c+s)^{(\eta(s-1))}, (c+\eta s)^{(c)} \right\}. $$

Therefore \(\rev{\mm_L}(G) = \max\{\eta-1,\eta(s-1),c\}.\)

If \(c=1\), then \(D=\eta(s-1),\) and hence \(\rev{\mm_L}(G)=D\).

Suppose now that \(c\geq2\). Then \(\eta(s-1)<D.\)

Furthermore, \(D-c=\eta(s-1)-1\geq1\) and \(D-(\eta-1)=c+\eta(s-2)>0.\)

Thus every Laplacian eigenvalue has multiplicity strictly smaller than \(D\). Hence \(\rev{\mm_L}(G)=\dim(G)\) if and only if \(c=1\).

We next consider the signless Laplacian matrix. The core \(C\), when \(c\geq2\), induces the signless Laplacian eigenvalue \(\alpha=c+\eta s-2\) with contribution \(c-1\). Each satellite \(V_i\) induces \(\beta=c+s-2\) with contribution \(s-1\), so the combined contribution of all satellites
to the multiplicity of \(\beta\) is \(\eta(s-1).\)

By Proposition~\ref{BlowUpFactorization}, the symmetric normalized signless Laplacian quotient matrix is $$ \widehat B_Q= \begin{pmatrix} \eta s+2c-2 &
\sqrt{cs} & \sqrt{cs} & \cdots & \sqrt{cs} \\ \sqrt{cs} & c+2s-2 & 0 & \cdots & 0 \\ \sqrt{cs} & 0 & c+2s-2 & \cdots & 0 \\ \vdots & \vdots & \vdots & \ddots & \vdots
\\ \sqrt{cs} & 0 & 0 & \cdots & c+2s-2 \end{pmatrix}. $$

The subspace of vectors whose core coordinate is zero and whose satellite coordinates sum to zero has dimension \(\eta-1\). Hence \(\gamma=c+2s-2\) is a quotient eigenvalue of multiplicity \(\eta-1\).

The two remaining quotient eigenvalues are the eigenvalues of $$ R = \begin{pmatrix} \eta s+2c-2 & \sqrt{\eta cs}\\ \sqrt{\eta cs} & c+2s-2 \end{pmatrix}. $$ Since its off-diagonal entry is nonzero, these two eigenvalues are distinct.

We now determine the possible collisions between the twin-induced and quotient eigenvalues.

First, \(\alpha=\gamma\) if and only if \(\eta=2.\) In this case the multiplicity of \(\alpha\) is \((c-1)+(\eta-1)=c.\)

Next, a direct calculation gives \(\det(\beta I-R) = s(\eta-1)(s-c).\)

Thus \(\beta\) is an eigenvalue of \(R\) if and only if \(c=s.\)

Since the eigenvalues of \(R\) are simple, this collision contributes exactly one additional dimension. Consequently, the multiplicity of \(\beta\) is \(\eta(s-1)+1\) when \(c=s\), and \(\eta(s-1)\) otherwise.

There are no further collisions. Indeed, \(\det(\alpha I-R) = -2cs(\eta-1)\ne0,\) and \(\det(\gamma I-R) = -\eta cs\ne0.\)

Suppose first that \(c=1\). Then \(D=\eta(s-1),\) and \(\beta=s-1\) has multiplicity \(D\). All the remaining multiplicities are smaller, so \(\rev{\mm_Q}(G)=D.\)

Now suppose that \(c\geq2\). If \(c\ne s\), then the multiplicity of \(\beta\) is \(\eta(s-1)=D-(c-1)<D.\)

If \(c=s\), then its multiplicity is \(\eta(s-1)+1 = D-(c-2).\)

Hence it equals \(D\) precisely when \(c=s=2.\)

The eigenvalue \(\alpha\) has multiplicity \(c-1\), except when \(\eta=2\), in which case its multiplicity is \(c\). In either case its multiplicity is strictly smaller than \(D\). Likewise, \(\eta-1<D,\) and the remaining quotient eigenvalues are simple.

Therefore \(\rev{\mm_Q}(G)=\dim(G)\) if and only if \(c=1 \qquad\text{or}\qquad (c,s)=(2,2).\)

\end{proof}

\begin{corollary}\label{CoreSatelliteLQSeparation} \rev{\rev{For every \(\eta\in\mathbb{Z}_{\geq 2}\)}, the graph} \(G=K_2\vee(\eta K_2)\) satisfies \(\rev{\mm_Q}(G)=\dim(G)=\eta+1,\) whereas \(\rev{\mm_L}(G)=\eta<\eta+1=\dim(G).\) \rev{Consequently, \(\{K_2\vee(\eta K_2):\eta\in\mathbb{Z}_{\geq2}\}\) is an infinite family} for which the signless Laplacian equality holds but the Laplacian equality does not. \end{corollary}

\begin{remark} \rev{When \(c=1\)}, Theorem~\ref{CoreSatelliteEquality} gives \(\rev{\mm_L}(G)=\rev{\mm_Q}(G)=\dim(G)=\eta(s-1).\) \end{remark}

\section{Core--Leaf Graphs}\label{sec:leaf}

Fix a connected graph \(H\) with exactly \(h\) vertices \(v_1,\ldots,v_h\), where \(h\in\mathbb{Z}_{\geq1}\). For every choice of \(p_1,\ldots,p_h\in\mathbb{Z}_{\geq2}\), let \(G=G(H;p_1,\ldots,p_h)\) be obtained by adding \(p_i\) new vertices adjacent only to \(v_i\), for every \(i\in\{1,\ldots,h\}\). Such new vertices are called pendant leaves. The \emph{core} is the entire starting graph \(H\), regarded as the induced subgraph on the original vertices of \(G\). Leaves are therefore attached to every vertex of \(H\), not merely to a selected part of it. These assumptions and the notation \(G\) are retained throughout this section\revblue{, except for the explicitly stated relaxation in Proposition~\ref{CoreLeafOneSingleLeaf}}.

The metric dimension of graphs with pendant edges was studied by
Iswadi et al.~\cite[Theorem~5]{Iswadi2008}; their result gives
\[
\dim(G(H;p_1,\ldots,p_h))
=
\sum_{i=1}^{h}(p_i-1)
.
\]

For a rooted-product extension, see Farhan and Baskoro~\cite[Proposition~3.3]{FarhanBaskoroFurther2026}. Faria~\cite{Faria1985} established the pendant/quasi-pendant lower bound for the multiplicity of the Laplacian eigenvalue \(1\). Andrade et al.~\cite{AndradeCardosoPastenRojo2015} proved, for both \(L\) and \(Q\), that the multiplicity is exactly the number of pendant vertices minus the number of their distinct neighbors when every nonpendant vertex is adjacent to a pendant vertex. This applies here and gives \(\mult_{L(G)}(1)=\mult_{Q(G)}(1)=\sum_i(p_i-1)\). We include a self-contained verification, then bound every other multiplicity to show maximality and, except for \(P_3\), uniqueness of the maximum.

\begin{theorem}\label{CoreLeafEquality}
\rev{The graph constructed above satisfies}

\[
\rev{\mm_L}(G)=\rev{\mm_Q}(G)=\dim(G)
=\sum_{i=1}^{h}(p_i-1),
\]
\rev{and the maximum multiplicity of each of \(L(G)\) and \(Q(G)\) is attained at the eigenvalue \(1\).}
\end{theorem}

\begin{proof}
The metric-dimension formula is known from
\cite[Theorem~5]{Iswadi2008} for \(h\geq2\); we include a short proof
valid also for \(h=1\). Let \(C_i\) be the leaves at \(v_i\), choose
\(S\) by omitting exactly one leaf \(u_i\) from each \(C_i\), and put
\(\rev{\ell}=\sum_i(p_i-1)\). The twin bound gives \(\dim(G)\geq \rev{\ell}\).
Choose \(x_i\in S\cap C_i\), which is possible since \(p_i\geq2\).
Distinct core vertices \(v_i,v_j\) are distinguished by \(x_i\), and
distinct omitted leaves \(u_i,u_j\) are distinguished by \(x_i\).
A core vertex \(v_j\) and any omitted leaf \(u_i\) are distinguished
by \(x_j\): their distances are \(1\) and \(2\) if \(i=j\), and
\(1\) and \(d_H(v_i,v_j)+2\) otherwise.
Vertices in \(S\) are distinguished by their zero coordinates.
Thus \(S\) resolves \(G\), and \(\dim(G)=\rev{\ell}\).

For each \(i\), the \(p_i\) leaves adjacent to \(v_i\) form a false
twin class of degree \(1\). Hence \(1\) is an eigenvalue of both
\(L(G)\) and \(Q(G)\) with multiplicity at least \(\rev{\ell}\).

Let \(\mathcal U\) be the subspace of vectors that are constant on
each leaf class. Write \(a_i\) for the common coordinate on the leaves
adjacent to \(v_i\) and \(b_i\) for the coordinate at \(v_i\). If
\(L(G)x=x\), then the leaf and core equations give successively
\(b_i=0\) and \(p_i a_i=0\). The same conclusion follows from
\(Q(G)x=x\). Thus \(1\) does not occur in either restriction to
\(\mathcal U\), and therefore
\[
\mult_{L(G)}(1)=\mult_{Q(G)}(1)=\rev{\ell}.
\]

\rev{For every eigenvalue \(\lambda\ne1\), the leaf equations} give
\[
a_i=\frac{b_i}{1-\lambda}\quad\text{for }L(G),
\qquad
a_i=-\frac{b_i}{1-\lambda}\quad\text{for }Q(G).
\]
Thus every such eigenvector is determined by its \(h\) core
coordinates, so every eigenvalue different from \(1\) has
multiplicity at most \(h\). Since \(p_i\geq2\),
\[
\rev{\ell}=\sum_{i=1}^{h}(p_i-1)\geq h.
\]
Therefore \(1\) has maximum multiplicity for both matrices, and
\(
\rev{\mm_L}(G)=\rev{\mm_Q}(G)=\rev{\ell}=\dim(G).
\)
\end{proof}

\begin{remark}\label{CoreLeafUniqueMaximum}
For every core of order \(h\in\mathbb{Z}_{\geq2}\), and every permitted choice of the integer leaf counts, the eigenvalue \(1\) is uniquely of maximum multiplicity for each of \(L(G)\) and \(Q(G)\). To see this, fix \(X\in\{L,Q\}\) and an eigenvalue \(\lambda\ne1\). Put \(\ell=\sum_i(p_i-1)\), and let \(b=(b_1,\ldots,b_h)^T\) record an eigenvector's coordinates at the core vertices. Each coordinate at an attached leaf is determined by the coordinate at its neighbor in \(H\), by the leaf equations in the preceding proof. Substituting these values into the core equations gives
\[
\left(X(H)+\operatorname{diag}\left(
p_i-\lambda-\frac{p_i}{1-\lambda}:i\in\{1,\ldots,h\}
\right)\right)b=0.
\]
This matrix has a nonzero off-diagonal entry because \(H\) is connected and \(h\geq2\). Hence its rank is at least one and its nullity, the dimension of its kernel, is at most \(h-1\). The map from the original eigenspace to \(b\) is injective, so \(\mult_{X(G)}(\lambda)\leq h-1<\ell\). For the one-vertex core, the same uniqueness holds \rev{for every \(p_1\in\mathbb{Z}_{\geq 3}\)}. When \(p_1=2\), the graph is \(P_3\), and every eigenvalue of either matrix has multiplicity one.
\end{remark}

\renewcommand{\thetheorem}{\textcolor{black}{\arabic{theorem}}}
\begin{proposition}\label{CoreLeafOneSingleLeaf}
Let \(H\) be a connected graph with vertices \(v_1,\ldots,v_h\), where
\(h\in\mathbb{Z}_{\geq3}\). Choose \(s\in\{1,\ldots,h\}\) such that
\(\deg_H(v_s)\geq2\). Set \(p_s=1\), and choose
\(p_i\in\mathbb{Z}_{\geq2}\) for every \(i\ne s\).
Let \(G\) be obtained by attaching \(p_i\) pendant leaves to each \(v_i\),
and put \(\ell=\sum_{i=1}^{h}(p_i-1)\). Then
\[
\dim(G)=\mm_L(G)=\mm_Q(G)=\ell.
\]
For each \(X\in\{L,Q\}\), the eigenvalue \(1\) has multiplicity \(\ell\)
and is the unique eigenvalue of maximum multiplicity. More precisely,
every eigenvalue \(\lambda\ne1\) of \(X(G)\) has multiplicity at most
\(h-2<\ell\).
\end{proposition}

\begin{proof}
Let \(C_i\) be the set of leaves attached to \(v_i\), choose \(u_i\in C_i\),
and let \(S=\bigcup_{i=1}^{h}(C_i\setminus\{u_i\})\).
The twin bound gives \(\dim(G)\geq\ell=|S|\).
For every \(j\ne s\), choose \(x_j\in C_j\setminus\{u_j\}\). For all
\(i\in\{1,\ldots,h\}\),
\[
d_G(x_j,v_i)=1+d_H(v_j,v_i),\qquad
d_G(x_j,u_i)=2+d_H(v_j,v_i).
\]
Distinct core vertices are distinguished by \(x_j\) at one of their
indices different from \(s\); the same argument distinguishes distinct
omitted leaves. Also, \(x_j\) distinguishes \(v_j\) from every omitted
leaf when \(j\ne s\), and any \(x_j\) distinguishes \(v_s\) from \(u_s\).
It remains to distinguish \(v_s\) from \(u_i\) for \(i\ne s\).
If \(d_H(v_s,v_i)\geq2\), then \(x_i\) has distance at least \(3\) from
\(v_s\) and distance \(2\) from \(u_i\). Otherwise, choose a neighbor
\(v_j\ne v_i\) of \(v_s\), using \(\deg_H(v_s)\geq2\). Now
\(d_G(x_j,v_s)=2\), whereas \(d_G(x_j,u_i)\geq3\).
Pairs meeting \(S\) are distinguished by a zero coordinate.
Thus \(S\) resolves \(G\), and \(\dim(G)=\ell\).

Fix \(X\in\{L,Q\}\). The leaf equations for an eigenvector of eigenvalue
\(1\) force all core coordinates to vanish. The core equations then force
the coordinates in each \(C_i\) to sum to zero. These conditions also
suffice, so \(\mult_{X(G)}(1)=\ell\).

For \(\lambda\ne1\), the leaf equations determine all leaf coordinates
from the core coordinates \(b=(b_1,\ldots,b_h)^T\).
Substitution into the core equations gives
\[
R_X(\lambda)b=0,\qquad
R_X(\lambda)=X(H)+\operatorname{diag}\left(
p_i-\lambda-\frac{p_i}{1-\lambda}:i\in\{1,\ldots,h\}\right).
\]
Restriction to the core coordinates is injective, so it suffices to prove
\(\rank R_X(\lambda)\geq2\).
The rank is nonzero because \(H\) has an edge. Suppose that the rank is
one, and write \(R=R_X(\lambda)=czz^T\), where \(c\ne0\).
Every vertex of \(H\) has a neighbor, so every coordinate of \(z\) is
nonzero. Consequently all off-diagonal entries are nonzero, forcing
\(H=K_h\). These entries are all equal to
\(\varepsilon\), where \(\varepsilon=-1\) for \(X=L\) and
\(\varepsilon=1\) for \(X=Q\). Since \(h\geq3\), the rank-one identity
\(R_{ii}=R_{ij}R_{ik}/R_{jk}\), for distinct \(i,j,k\), shows that
every diagonal entry also equals \(\varepsilon\).
Comparing the diagonal entries at \(s\) and any \(i\ne s\) now gives
\[
0=(p_s-p_i)\left(1-\frac{1}{1-\lambda}\right)
=-\frac{\lambda(p_s-p_i)}{1-\lambda}.
\]
Since \(p_s\ne p_i\), we obtain \(\lambda=0\). But
\(R_X(0)=X(K_h)\), whose rank is \(h-1\) for \(X=L\) and \(h\) for
\(X=Q\), a contradiction.
Thus every other multiplicity is at most \(h-2\), while
\(\ell\geq h-1\). This proves the asserted equalities and uniqueness.
\end{proof}

\subsection{Enumeration of core--leaf equality graphs}

The family considered in Theorem~\ref{CoreLeafEquality} also admits a natural enumerative description. We restrict attention to the graphs arising in that theorem, \rev{in which the whole connected starting graph is the core and at least two pendant leaves are added at every one of its vertices.}

\rev{\rev{For every \(n\in\mathbb{Z}_{\geq 3}\)}, let \(\mathcal B_n\) be the set of isomorphism classes of graphs of order \(n\) from this construction, allowing all connected cores and all permitted leaf counts, and put \(b_n=|\mathcal B_n|\).}

By Theorem~\ref{CoreLeafEquality}, every graph counted by \(b_n\) satisfies \(\rev{\mm_L}(G)=\rev{\mm_Q}(G)=\dim(G).\)

First fix a core \(H\). Its automorphism group \(\operatorname{Aut}(H)\) consists of the adjacency-preserving permutations of \(V(H)\); its action on vertices is \((\sigma,v)\mapsto\sigma(v)\). For every automorphism \(\sigma\) and every \(j\in\{1,\ldots,h\}\), let \(c_j(\sigma)\) count its permutation cycles of length \(j\), including fixed vertices as cycles of length one. Define the \emph{cycle-index polynomial}
\[
Z_H(z_1,\ldots,z_h)=\frac{1}{|\operatorname{Aut}(H)|}
\sum_{\sigma\in\operatorname{Aut}(H)}\prod_{j=1}^h z_j^{c_j(\sigma)}.
\]
\rev{For every \(n\in\mathbb{Z}_{\geq 3h}\)}, let \(b_{H,n}\) count the isomorphism classes of order \(n\) obtained from this fixed core.

\begin{theorem}\label{CoreLeafUnlabeledEnumeration} \rev{For the fixed core and the counts just defined,} $$ B_H(x):=\sum_{n\geq3h}b_{H,n}x^n = \frac{x^{3h}}{|\operatorname{Aut}(H)|} \sum_{\sigma\in\operatorname{Aut}(H)} \prod_{j\geq1} (1-x^j)^{-c_j(\sigma)}. $$ \rev{Equivalently, using the cycle-index polynomial defined above,} $$ B_H(x) = x^{3h} \rev{Z_H} \left( \frac{1}{1-x}, \frac{1}{1-x^2}, \ldots, \frac{1}{1-x^h} \right). $$ Consequently, $$ B(x):=\sum_{n\geq3}b_nx^n = \rev{\sum_{[H]}} B_H(x), $$ where the sum ranges over isomorphism classes of connected graphs. \end{theorem}

\begin{proof}

Since at least two pendant vertices are attached to every core vertex, the vertices of the core are precisely the nonpendant vertices of the resulting graph. \rev{Thus deleting every degree-one vertex of \(G\) leaves precisely the induced copy of \(H\). This is what it means here for the core to be uniquely recoverable, up to isomorphism.}

For a fixed \(H\), an element of the family is therefore determined,
up to isomorphism, by a function \(p:V(H)\longrightarrow\{2,3,\ldots\},\) where \(p(v)\) is the number of pendant leaves attached to \(v\), \rev{modulo the action \((\sigma\cdot p)(v)=p(\sigma^{-1}(v))\). Thus two assignments are identified exactly when an automorphism of \(H\) carries one to the other.}

We apply Burnside's lemma. Suppose that \(\sigma\in\operatorname{Aut}(H)\). An assignment \(p\) is fixed by \(\sigma\) precisely when it is constant on every cycle of \(\sigma\). \rev{For every permutation cycle of length \(j\), assigning a common leaf count \(p\in\mathbb{Z}_{\geq2}\) contributes} $$ \sum_{p\geq2}x^{jp} = \frac{x^{2j}}{1-x^j}. $$

Hence the generating function for assignments fixed by \(\sigma\) is $$ x^h \prod_{j\geq1} \left( \frac{x^{2j}}{1-x^j} \right)^{c_j(\sigma)}. $$

Since $$ \sum_{j\geq1}j\,c_j(\sigma)=h, $$ this simplifies to $$ x^{3h} \prod_{j\geq1}(1-x^j)^{-c_j(\sigma)}. $$

Averaging over \(\operatorname{Aut}(H)\) gives the claimed formula by Burnside's lemma. Summing over all connected cores gives \(B(x)\).

\end{proof}

For a core of order \(h\), every resulting graph has
\(n=h+\sum_i p_i\) and \(\ell=\dim(G)=\sum_i(p_i-1)=n-2h\).
Let \(b_H(n,\ell)\) count the resulting isomorphism classes of order \(n\) and metric dimension \(\ell\) for the fixed core \(H\). Its generating function is
\[
B_H(x,y)=\sum_{n,\ell}b_H(n,\ell)x^ny^\ell
=y^{-2h}B_H(xy)
=\frac{x^{3h}y^h}{|\operatorname{Aut}(H)|}
\sum_{\sigma\in\operatorname{Aut}(H)}
\prod_{j\geq1}(1-(xy)^j)^{-c_j(\sigma)}.
\]
For the fixed graph \(H\), this is a change of variables in Theorem~\ref{CoreLeafUnlabeledEnumeration}, not a separate enumeration. When cores of different orders are combined, the refinement is useful because \((n,\ell)\) determines \(h=(n-\ell)/2\). Let \(b(n,\ell)\) count these classes over all connected cores. Accordingly, put
\[
B(x,y)=\sum_{n,\ell}b(n,\ell)x^ny^\ell
=\rev{\sum_{[H]}}B_H(x,y).
\]

We next obtain an asymptotic form of the refined enumeration. \rev{\rev{For every \(h\in\mathbb{Z}_{\geq1}\)}, let \(c_h\) denote} the number of connected nonisomorphic graphs of order \(h\). \rev{A graph is asymmetric when its only automorphism is the identity. Saying that almost all unlabelled graphs have a property means that the fraction of isomorphism classes of order \(h\) having that property tends to one as the integer \(h\) tends to infinity. In this precise sense, almost all unlabelled graphs are asymmetric; see Babai~\cite[discussion preceding Theorem~1.20]{Babai1995}, which explicitly distinguishes the labelled and unlabelled models. Almost all unlabelled graphs are also connected~\cite{Wright1975}. Consequently, among connected unlabelled graphs of order \(h\), the proportion with a nonidentity automorphism tends to zero as \(h\to\infty\). This fact bounds the relative counting error in the proof of Theorem~\ref{CoreLeafAsymptotic}.}

\begin{theorem}\label{CoreLeafAsymptotic}

For every \(h\in\mathbb{Z}_{\geq1}\) and every \(\ell\in\mathbb{Z}_{\geq h}\), put \(n=\ell+2h\). As \(h\to\infty\),
\[
b(n,\ell)\sim c_h\binom{\ell-1}{h-1}.
\]
The symbol \(\sim\) means that the ratio tends to one. The estimate is uniform over \(\ell\geq h\): for every real \(\varepsilon>0\), there is an integer \(h_0\) such that, \rev{for every \(h\in\mathbb{Z}_{\geq h_0}\)} and every \(\ell\in\mathbb{Z}_{\geq h}\),
\[
\left|\frac{b(\ell+2h,\ell)}{c_h\binom{\ell-1}{h-1}}-1\right|<\varepsilon.
\]

\end{theorem}

\begin{proof}
\rev{Fix a core \(H\) of order \(h\), and write \(p_i=2+q_i\). For a graph of order \(n\), the integer \(s=n-3h=\ell-h\) counts the extra leaves beyond two per core vertex. The nonnegative integers \(q_i\) therefore satisfy \(q_1+\cdots+q_h=s\). The notation \(s\) is used only in this proof to count these assignments.}
If \(H\) is asymmetric, then \(\operatorname{Aut}(H)\) is trivial, so distinct assignments \((q_1,\ldots,q_h)\) yield nonisomorphic core--leaf graphs. Hence every asymmetric core contributes exactly
\(N_{h,s}=\binom{s+h-1}{h-1}\)
graphs.

Let \(r_h\) be the number of connected nonisomorphic graphs of order \(h\) with a nontrivial automorphism. The classical asymmetry and connectedness results cited above imply
\(r_h=o(c_h).\)
For a symmetric core, the number of orbits of the \(N_{h,s}\) assignments under its automorphism group is at most \(N_{h,s}\). Consequently,
\[
(c_h-r_h)N_{h,s}
\leq b(n,\ell)\leq
c_hN_{h,s}.
\]
Dividing by \(c_hN_{h,s}\) and letting \(h\to\infty\) gives the result. Since the error is bounded solely by \(r_h/c_h\), independently of \(s\), the estimate is uniform in \(s\). Finally,
\(s+h-1=\ell-1,\)
which gives the first form of the formula.
\end{proof}

\begin{corollary}\label{CoreLeafExponentialGrowth}
The number \(b_n\) of core--leaf graphs satisfies
\[
\lim_{n\to\infty}\frac{\log_2 b_n}{n^2}=\frac1{18}.
\]
If \(e_X(n)\) counts all connected nonisomorphic graphs \rev{\(G\)} of order \(n\)
satisfying \(\rev{\mm_X}(G)=\dim(G)\), then, \rev{for each choice \(X\in\{L,Q\}\)},
\[
\frac1{18}\leq\liminf_{n\to\infty}
\frac{\log_2 e_X(n)}{n^2}
\leq\limsup_{n\to\infty}
\frac{\log_2 e_X(n)}{n^2}\leq\frac12.
\]
\end{corollary}
\begin{proof}
Put \(h=\lfloor n/3\rfloor\) and \(r=n-3h\). Every connected core
of order \(h\) yields a graph of order \(n\) by attaching two leaves
to each vertex and adding \(r\) extra leaves at one vertex.
Different cores yield nonisomorphic graphs. Among labelled graphs on
\(h\) vertices, at least \(2^{\binom{h-1}{2}}\) are connected: fix
one vertex to be universal and choose the remaining edges freely.
Thus
\[
b_n\geq c_h\geq\frac{2^{\binom{h-1}{2}}}{h!}.
\]
Conversely, counting labelled cores and all leaf assignments overcounts
the isomorphism classes and gives
\[
b_n\leq\sum_{j=1}^{h}2^{\binom j2}
\binom{n-2j-1}{j-1}
\leq n\,2^{\binom h2+n}.
\]
Taking logarithms proves the limit. Finally,
\(b_n\leq e_X(n)\leq2^{\binom n2}\), proving the remaining bounds.
\end{proof}

The limit \(1/18\) is exact for the core--leaf subfamily. For the full equality classes, the upper bound \(1/2\) uses only the count of all graphs and does not exploit either equality condition.

For each \(n\in\{3,\ldots,17\}\), only cores with at most \(\lfloor n/3\rfloor\leq5\) vertices contribute. Write \([x^n]F(x)\) for the coefficient of \(x^n\) in a power series. For each \(h\in\{1,\ldots,5\}\), define the aggregate cycle-index polynomial
\[
Z_h(z_1,\ldots,z_h)=\sum_{[H]}Z_H(z_1,\ldots,z_h),
\]
where \([H]\) ranges over all isomorphism classes of connected graphs of order \(h\). The fixed-core formula gives
\[
b_n=[x^n]\sum_{h=1}^{5}x^{3h}
Z_h\bigl((1-x)^{-1},\ldots,(1-x^h)^{-1}\bigr).
\]
Enumerating these cores and their adjacency-preserving vertex permutations gives
\begin{align*}
Z_1&=z_1,\\
Z_2&=\tfrac12z_1^2+\tfrac12z_2,\\
Z_3&=\tfrac23z_1^3+z_1z_2+\tfrac13z_3,\\
Z_4&=\tfrac{19}{12}z_1^4+2z_1^2z_2+\tfrac54z_2^2
       +\tfrac23z_1z_3+\tfrac12z_4,\\
Z_5&=\tfrac{91}{15}z_1^5+\tfrac{19}{3}z_1^3z_2+5z_1z_2^2
       +\tfrac43z_1^2z_3+\tfrac23z_2z_3+z_1z_4+\tfrac35z_5.
\end{align*}
For orders \(h=1,2,3,4,5\), the numbers \(c_h\) of connected cores are respectively \(1,1,2,6,21\). Every connected graph is an admissible core. Thus, for \(h=2,3,4,5\), these are exactly the entries in the connected-graph column of Table~\ref{SmallGraphCensus}; the additional case \(h=1\) is \(K_1\). They can be generated by the vertex-addition procedure in Section~\ref{sec:census}; for each core, testing the \(h!\) permutations gives its automorphism group and hence its cycle-index polynomial. As a small example, the path and triangle of order three have cycle indices \((z_1^3+z_1z_2)/2\) and \((z_1^3+3z_1z_2+2z_3)/6\), whose sum is \(Z_3\). To calculate the table below, replace each variable \(z_j\) in \(Z_h(z_1,\ldots,z_h)\) by \((1-x^j)^{-1}\), for every \(j\in\{1,\ldots,h\}\). Multiply the resulting expression by \(x^{3h}\), sum over \(h=1,\ldots,5\), and expand through degree \(17\). The coefficient of \(x^n\) is \(b_n\).

\rev{The first values are the following.} $$ \begin{array}{c|rrrrrrrrrrrrrrr} n &3&4&5&6&7&8&9&10&11&12&13&14&15&16&17\\ \hline b_n &1&1&1&2&2&3&5&7&10&20&29&50&95&172&312 \end{array} $$

\section{\rev{A Census of Small Graphs}}\label{sec:census}

The preceding results describe several infinite families satisfying $$ \rev{\mm_L}(G)=\dim(G) \qquad\text{or}\qquad \rev{\mm_Q}(G)=\dim(G). $$

The counts in Table~\ref{SmallGraphCensus} concern the connected isomorphism classes of orders two through seven. They can be obtained by the following finite procedure, which specifies the graph enumeration and the tests independently of any software implementation.

Start with \(K_1\). For every \(n\in\{2,\ldots,7\}\), take one representative of each connected graph of order \(n-1\), add a new vertex with each nonempty subset of the existing vertices in turn as its neighborhood, and retain one representative of each resulting isomorphism class. No connected graph is missed: a leaf of any spanning tree can be deleted while leaving a connected graph on the remaining vertices. Two adjacency matrices represent the same class precisely when one is obtained from the other by applying the same permutation to its rows and columns. Thus even the removal of repetitions is a finite test. The numbers of representatives are \(1,2,6,21,112,853\), respectively, in agreement with the connected graphs in Read and Wilson's atlas~\cite{ReadWilson1998}.

For a representative with adjacency matrix \(A\), compute distances using
\[
 d(i,j)=\min\{t\in\{0,\ldots,n-1\}:(A^t)_{ij}>0\}.
\]
The set is nonempty because the graph is connected. Test vertex subsets \(S\) in increasing cardinality: \(S\) resolves the graph exactly when, for every pair \(i\ne j\), some \(s\in S\) satisfies \(d(i,s)\ne d(j,s)\). The first successful cardinality is therefore \(\dim(G)\).

Put \(D=\operatorname{diag}(\sum_j A_{1j},\ldots,\sum_j A_{nj})\). For each choice \(M\in\{D-A,D+A\}\), factor
\[
 \det(zI-M)=\prod_{i=1}^{t}f_i(z)^{a_i}
 \quad\text{over }\mathbb Q,
\]
where the \(f_i\) are distinct monic irreducible polynomials and the \(a_i\) are positive integers. Each \(f_i\) has distinct roots in characteristic zero, and different factors have no common root. Consequently the maximum eigenvalue multiplicity is exactly \(\max_i a_i\). Comparing these two maxima with \(\dim(G)\) gives the table. To identify the twin-free cases, test every distinct pair \(i,j\): they are twins exactly when \(A_{ik}=A_{jk}\) for every \(k\notin\{i,j\}\). These exhaustive finite tests establish the displayed counts; the proofs for the infinite families are given in their respective sections.

Table~\ref{SmallGraphCensus} reports the counts. The column \(\mm_L=\dim\) includes every graph satisfying the Laplacian equality, including those that also satisfy the signless Laplacian equality. In contrast, ``\(L\) only'' excludes those satisfying both. The two \(Q\) columns have the analogous meanings. Thus each equality column is the sum of ``both'' and its corresponding ``only'' column; for example, at order six, \(31=26+5\) and \(33=26+7\).

\begin{table}[H]
\centering
{
\captionsetup{labelfont={color=black},textfont={color=black}}
\begin{tabular}{c|r|r|r|r|r|r}
\(n\) & connected & \(\rev{\mm_L}=\dim\) & \(\rev{\mm_Q}=\dim\) & both & \(L\) only & \(Q\) only\\
\hline
2 & 1   & 1  & 1  & 1  & 0 & 0\\
3 & 2   & 2  & 2  & 2  & 0 & 0\\
4 & 6   & 5  & 5  & 5  & 0 & 0\\
5 & 21  & 8  & 8  & 8  & 0 & 0\\
6 & 112 & 31 & 33 & 26 & 5 & 7\\
7 & 853 & 41 & 66 & 32 & 9 & 34
\end{tabular}
\caption{Connected nonisomorphic graphs satisfying the Laplacian and signless Laplacian equality conditions.}
\label{SmallGraphCensus}
}
\end{table}

Several features of the table are noteworthy. First, for every connected graph \rev{with \(2\leq |V(G)|\leq5\)}, $$ \rev{\mm_L}(G)=\dim(G) \quad\Longleftrightarrow\quad \rev{\mm_Q}(G)=\dim(G). $$

Order \(6\) is the first order at which the two properties differ. There are five connected graphs of order \(6\) satisfying the Laplacian equality but not the signless Laplacian equality, and seven satisfying the signless Laplacian equality but not the Laplacian equality.

One of the latter graphs is \(K_2\vee(2K_2),\) the smallest member of the infinite family identified in Corollary~\ref{CoreSatelliteLQSeparation}. \rev{This gives an explicit graph of the first order at which the two equality conditions differ.}

There are also small examples in the opposite direction. \rev{Take \(T=\Theta_{2,2,3}\), with vertex set \(\{u,v,a,b,c,d\}\) and edge set \(\{ua,av,ub,bv,uc,cd,dv\}\). The three paths are \(u,a,v\), \(u,b,v\), and \(u,c,d,v\); they share only their endpoints \(u,v\).} A direct calculation gives \(\dim(T)=2\) and \(\chi_{L(T)}(x) = x(x-4)(x-2)^2(x^2-6x+6),\) so that \(\rev{\mm_L}(T)=2.\) On the other hand, $$ \chi_{Q(T)}(x) = (x-2)(x^2-4x+2)(x^3-8x^2+16x-4), $$ which is square-free. Hence \(\rev{\mm_Q}(T)=1,\) and therefore \(\rev{\mm_L}(T)=\dim(T)>\rev{\mm_Q}(T).\)

The census also shows that equality is not exclusively a twin-induced phenomenon. For example, paths of order at least \(4\) have no nontrivial twin classes, while \(\rev{\mm_L}(P_n)=\rev{\mm_Q}(P_n)=\dim(P_n)=1.\)

More generally, among the connected graphs of order \(7\), there are twin-free graphs for which exactly one of the two equality conditions holds: \rev{three satisfy only the Laplacian equality and six satisfy only the signless Laplacian equality, as verified by the pairwise twin test on every representative in the census; Table~\ref{TwinFreeExamples} lists all nine.} Thus a complete characterization of the equality cases must necessarily account for spectral multiplicities not arising from twin classes.

Table~\ref{TwinFreeExamples} specifies all nine graphs by their edge sets on the vertex set \(\{1,\ldots,7\}\).
\begin{table}[H]
\centering
\color{black}\small
\captionsetup{labelfont={color=black},textfont={color=black}}
\renewcommand{\arraystretch}{1.15}
\begin{tabular}{c|>{\raggedright\arraybackslash}p{0.64\textwidth}|ccc}
 & Edges & \(\dim(G)\)&\(\mm_L(G)\)&\(\mm_Q(G)\)\\\hline
1 & \(\{1,2\},\allowbreak \{1,6\},\allowbreak \{1,7\},\allowbreak \{2,3\},\allowbreak \{2,6\},\allowbreak \{3,4\},\allowbreak \{4,5\},\allowbreak \{5,6\}\) & 2&2&1\\[3pt]
2 & \(\{1,2\},\allowbreak \{1,6\},\allowbreak \{1,7\},\allowbreak \{2,3\},\allowbreak \{3,4\},\allowbreak \{4,5\},\allowbreak \{4,7\},\allowbreak \{5,6\}\) & 2&2&1\\[3pt]
3 & \(\{1,2\},\allowbreak \{1,5\},\allowbreak \{2,3\},\allowbreak \{2,7\},\allowbreak \{3,4\},\allowbreak \{4,5\},\allowbreak \{4,6\},\allowbreak \{5,6\},\allowbreak \{6,7\}\) & 3&3&2\\[3pt]
\hline
4 & \(\{1,2\},\allowbreak \{1,5\},\allowbreak \{1,6\},\allowbreak \{2,3\},\allowbreak \{2,6\},\allowbreak \{2,7\},\allowbreak \{3,4\},\allowbreak \{3,6\},\allowbreak \{4,5\}\) & 2&1&2\\[3pt]
5 & \(\{1,2\},\allowbreak \{1,4\},\allowbreak \{1,5\},\allowbreak \{2,3\},\allowbreak \{2,5\},\allowbreak \{2,6\},\allowbreak \{2,7\},\allowbreak \{3,4\},\allowbreak \{3,5\},\allowbreak \{3,6\},\allowbreak \{5,6\},\allowbreak \{6,7\}\) & 2&1&2\\[3pt]
6 & \(\{1,4\},\allowbreak \{1,6\},\allowbreak \{2,3\},\allowbreak \{2,6\},\allowbreak \{2,7\},\allowbreak \{3,5\},\allowbreak \{3,7\},\allowbreak \{4,5\},\allowbreak \{4,7\},\allowbreak \{5,6\},\allowbreak \{5,7\},\allowbreak \{6,7\}\) & 2&1&2\\[3pt]
7 & \(\{1,2\},\allowbreak \{1,3\},\allowbreak \{1,4\},\allowbreak \{1,5\},\allowbreak \{1,6\},\allowbreak \{1,7\},\allowbreak \{2,3\},\allowbreak \{2,4\},\allowbreak \{2,5\},\allowbreak \{3,6\},\allowbreak \{3,7\},\allowbreak \{4,6\},\allowbreak \{5,7\}\) & 3&2&3\\[3pt]
8 & \(\{1,3\},\allowbreak \{1,6\},\allowbreak \{1,7\},\allowbreak \{2,3\},\allowbreak \{2,5\},\allowbreak \{2,7\},\allowbreak \{3,4\},\allowbreak \{4,5\},\allowbreak \{4,6\},\allowbreak \{4,7\},\allowbreak \{5,6\},\allowbreak \{5,7\},\allowbreak \{6,7\}\) & 3&2&3\\[3pt]
9 & \(\{1,2\},\allowbreak \{1,3\},\allowbreak \{1,4\},\allowbreak \{1,7\},\allowbreak \{2,5\},\allowbreak \{2,6\},\allowbreak \{2,7\},\allowbreak \{3,4\},\allowbreak \{3,6\},\allowbreak \{3,7\},\allowbreak \{4,5\},\allowbreak \{4,7\},\allowbreak \{5,6\},\allowbreak \{5,7\},\allowbreak \{6,7\}\) & 3&2&3\\[3pt]
\end{tabular}
\caption{The nine twin-free separating graphs of order seven, specified by edges.}
\label{TwinFreeExamples}
\end{table}

If \(G\) is bipartite, put \(+1\) on one bipartition class and \(-1\) on the other as the diagonal entries of a matrix \(S\). Then \(S^{-1}=S\) and \(Q(G)=S^{-1}L(G)S\), so the matrices are similar and have identical spectra, including multiplicities. Hence \(\mm_L(G)=\mm_Q(G)\), and every graph separating the two equality properties is nonbipartite. All separating examples in the census satisfy this requirement.

\section{A Family of Theta Graphs}\label{sec:theta}

The computational census of the preceding section suggests that theta graphs provide another natural setting in which the Laplacian and signless Laplacian equality properties may behave differently. We consider here a particularly simple infinite family.

\rev{For every \(p,q,r\in\mathbb{Z}_{\geq2}\), construct \(\Theta_{p,q,r}\) from three paths with \(p,q,r\) edges, respectively, by identifying their three initial vertices as one vertex \(u\) and their three terminal vertices as one distinct vertex \(v\). No other vertices are identified, and no additional edges are added. Thus the paths have exactly the two endpoints in common.}

\rev{Throughout this section, \rev{for every odd \(r\in\mathbb{Z}_{\geq 3}\)}, write \(G_r=\Theta_{2,2,r}\).}

The metric dimension of theta graphs was determined by Knor, Sedlar, and \v{S}krekovski~\cite{Knor2022}; their classification gives \(\dim(G_r)=2\) for the present family. We retain a direct verification and focus on the maximum eigenvalue multiplicities.

\rev{The two internal vertices of the length-two paths have the same open neighborhood \(\{u,v\}\) and are nonadjacent. They are therefore false twin vertices of degree \(2\),} and therefore already produce the eigenvalue \(2\) for both \(L(G_r)\) and \(Q(G_r)\). The following result shows that a second occurrence of this eigenvalue alternates between the Laplacian and signless Laplacian according to \(r\) modulo \(4\).

\begin{theorem}\label{ThetaAlternation} \rev{\rev{For every odd \(r\in\mathbb{Z}_{\geq 3}\)},} \(\dim(G_r)=2.\) Moreover: \begin{enumerate}[label=\textup{(\roman*)}] \item if \(r\equiv3\pmod4,\) then \(\rev{\mm_L}(G_r)=2\) \rev{and} \(\rev{\mm_Q}(G_r)=1\); \item if \(r\equiv1\pmod4,\) then \(\rev{\mm_L}(G_r)=1\) \rev{and} \(\rev{\mm_Q}(G_r)=2\). \end{enumerate} Consequently, \(\rev{\mm_L}(G_r)=\dim(G_r)\) if and only if \(r\equiv3\pmod4,\) whereas \(\rev{\mm_Q}(G_r)=\dim(G_r)\) if and only if \(r\equiv1\pmod4.\) \end{theorem}

\begin{proof}

Let \(a\) and \(b\) be the internal vertices of the two paths of length \(2\), and write \rev{\(r=2k+1\) for the positive integer \(k=(r-1)/2\).} Denote the vertices of the third path by \(u=w_0,w_1,\ldots,w_{2k+1}=v.\) We first verify the metric dimension directly. No single vertex resolves \(G_r\): if the chosen vertex is outside \(\{a,b\}\), then the false twins \(a\) and \(b\) have the same distance from it, whereas if the chosen vertex is \(a\) or \(b\), then \(u\) and \(v\) both have distance \(1\) from it. Hence \(\dim(G_r)\geq2\).

The set \(\{a,w_1\}\) resolves \(G_r\). \rev{When \(r=3\), this is checked immediately.} Suppose \(r\geq5\). For the internal long-path vertices, the distance pairs with respect to \((a,w_1)\) are \((i+1,i-1),\quad 1\leq i\leq k,\) \((k+1,k),\quad i=k+1,\) \((k,k+1),\quad i=k+2,\) and \((2k+2-i,\,2k+4-i),\quad k+3\leq i\leq2k.\) These four types are pairwise distinct. The remaining vertices have representations $$ r(a\mid\{a,w_1\})=(0,2),\qquad r(b\mid\{a,w_1\})=(2,2), $$ $$ r(u\mid\{a,w_1\})=(1,1),\qquad r(v\mid\{a,w_1\})=(1,3), $$ which are distinct from one another and from the preceding pairs. Thus \(\dim(G_r)=2\).

Since \(a\) and \(b\) are false twins of degree \(2\), \(e_a-e_b\) is an eigenvector with eigenvalue \(2\) for both \(L(G_r)\) and \(Q(G_r)\).

Consider the orthogonal complement of \(\operatorname{span}\{e_a-e_b\}\). On this subspace the coordinates at \(a\) and \(b\) are equal. The reflection \(w_i\longleftrightarrow w_{2k+1-i}\) which interchanges \(u\) and \(v\) and fixes \(a\) and \(b\) commutes with both \(L(G_r)\) and \(Q(G_r)\). Hence this complementary subspace splits into its symmetric and antisymmetric parts.

Write \(e_v\) for the coordinate vector at \(v\), and set
\[
q=\frac{e_a+e_b}{\sqrt2},\qquad
u_i^\pm=\frac{e_{w_i}\pm e_{w_{2k+1-i}}}{\sqrt2},\quad 0\leq i\leq k.
\]
The antisymmetric and symmetric bases are, respectively,
\[
(u_0^-,\ldots,u_k^-),\qquad (q,u_0^+,\ldots,u_k^+).
\]
Write \(T(a;b)\) for the symmetric tridiagonal matrix with diagonal \(a=(a_1,\ldots,a_m)\) and off-diagonal vector \(b=(b_1,\ldots,b_{m-1})\), so \(T_{i,i+1}=T_{i+1,i}=b_i\). In these bases the restrictions are
\[
\begin{aligned}
B_L^-&=T(3,\underbrace{2,\ldots,2}_{k-1},3;
                \underbrace{-1,\ldots,-1}_{k}),\\
B_L^+&=T(2,3,\underbrace{2,\ldots,2}_{k-1},1;
                -2,\underbrace{-1,\ldots,-1}_{k}),\\
B_Q^-&=T(3,\underbrace{2,\ldots,2}_{k-1},1;
                \underbrace{1,\ldots,1}_{k}),\\
B_Q^+&=T(2,3,\underbrace{2,\ldots,2}_{k-1},3;
                2,\underbrace{1,\ldots,1}_{k}).
\end{aligned}
\]
An empty repeated list is omitted, so these definitions apply also when \(k=1\). The entries \(\pm2\) follow from the four edges between \(\{a,b\}\) and \(\{u,v\}\) in the normalized symmetric basis.

All four matrices are irreducible symmetric tridiagonal matrices. Consequently, each of them has simple spectrum.

Let $$ f_X^\pm(z)=\det(zI-B_X^\pm), \qquad X\in\{L,Q\}. $$

We give an explicit continuant argument for disjointness of the two
spectra. Put \(t=z-2\), and define
\[
p_0=1,\qquad p_1=t,\qquad p_{j+1}=tp_j-p_{j-1}\quad\rev{\text{for every }j\in\mathbb{Z}_{\geq1}}.
\]
\rev{For a tridiagonal matrix \(T(a;b)\), the leading characteristic determinants satisfy \(D_0=1\), \(D_1=z-a_1\), and \(D_j=(z-a_j)D_{j-1}-b_{j-1}^2D_{j-2}\). Applying this recurrence to the four restrictions gives, \rev{\rev{for every \(k\in\mathbb{Z}_{\geq 1}\)}},}
\[
\begin{aligned}
f_L^-&=(t-2)p_k,\\
f_L^+&=(t+2)\bigl((t-2)p_k-2p_{k-1}\bigr),\\
f_Q^-&=tp_k-2p_{k-1},\\
f_Q^+&=(t^2-2t-4)p_k+4p_{k-1}.
\end{aligned}
\]
These identities also cover the smallest case \(k=1\).
Consecutive \(p_j\) have no common root, by the recurrence and
\(p_0=1\), and \(p_j(2)=j+1\),
\(p_j(-2)=(-1)^j(j+1)\).
A zero of \(f_L^-\) either has \(t=2\), where
\(f_L^+=-8k\ne0\), or satisfies \(p_k=0\), where
\(f_L^+=-2(t+2)p_{k-1}\ne0\).
\rev{In the signless Laplacian case, a zero} of \(f_Q^-\) satisfies
\(2p_{k-1}=tp_k\) and \(p_k\ne0\); hence at that zero
\(f_Q^+=(t^2-4)p_k\).
A common zero would therefore have \(t=2\) or \(t=-2\), but
\(tp_k-2p_{k-1}=2\) at \(t=2\) and
\(tp_k-2p_{k-1}=2(-1)^{k+1}\) at \(t=-2\).
Thus the symmetric and antisymmetric spectra are disjoint for both
matrices. Since each restriction has simple spectrum, every eigenvalue
in the complement of \(\operatorname{span}\{e_a-e_b\}\) is simple.

It remains only to determine whether the twin-induced eigenvalue \(2\) also occurs in this complementary component. Evaluating the same continuant recurrences at \(z=2\) gives $$ f_L^-(2)=0 \quad\Longleftrightarrow\quad k\text{ is odd}, $$ while \(f_L^+(2) \neq 0\) \rev{\rev{for every \(k\in\mathbb{Z}_{\geq 1}\)}}. Therefore \(2\) occurs once in the complementary Laplacian spectrum precisely when \(k\) is odd. Since \(r=2k+1 ,\) this is equivalent to \(r\equiv3\pmod4.\)

Together with the twin-induced eigenvector \(e_a-e_b\), this gives $$ \mult_{L(G_r)}(2) = \begin{cases} 2,&r\equiv3\pmod4,\\ 1,&r\equiv1\pmod4. \end{cases} $$

Since all remaining eigenvalues are simple, $$ \rev{\mm_L}(G_r) = \begin{cases} 2,&r\equiv3\pmod4,\\ 1,&r\equiv1\pmod4. \end{cases} $$

For the signless Laplacian, the corresponding evaluation gives $$ f_Q^-(2)=0 \quad\Longleftrightarrow\quad k\text{ is even}, $$ whereas \(f_Q^+(2) \neq 0\) \rev{\rev{for every \(k\in\mathbb{Z}_{\geq 1}\)}}. Thus $$ \mult_{Q(G_r)}(2) = \begin{cases} 1,&r\equiv3\pmod4,\\ 2,&r\equiv1\pmod4. \end{cases} $$

Again all remaining eigenvalues are simple, and hence $$ \rev{\mm_Q}(G_r) = \begin{cases} 1,&r\equiv3\pmod4,\\ 2,&r\equiv1\pmod4. \end{cases} $$

Combining these formulae with \(\dim(G_r)=2\) proves the result.

\end{proof}

\begin{corollary}\label{ThetaInfiniteSeparation} The two families $$ \rev{\{\Theta_{2,2,4j+3}\}_{j=0}^{\infty}} \qquad\text{and}\qquad \rev{\{\Theta_{2,2,4j+5}\}_{j=0}^{\infty}} $$ give infinitely many examples separating the two equality properties. \rev{More precisely, for every \(j\in\mathbb{Z}_{\geq0}\),} $$ \rev{\mm_L}(\Theta_{2,2,4j+3}) = \dim(\Theta_{2,2,4j+3}) = 2 > \rev{\mm_Q}(\Theta_{2,2,4j+3}) = 1 , $$ whereas $$ \rev{\mm_Q}(\Theta_{2,2,4j+5}) = \dim(\Theta_{2,2,4j+5}) = 2 > \rev{\mm_L}(\Theta_{2,2,4j+5}) = 1. $$ \end{corollary}

\begin{remark}
The graph \(\Theta_{2,2,3}\) is one of the smallest separating examples in the census. The formula for \(\mult_{Q(G_r)}(2)\) is a direct specialization of Paone and Paone~\cite[Corollary~8.1]{PaonePaone2026}, rather than a new value for that particular multiplicity. The argument above additionally proves that all other signless Laplacian eigenvalues are simple, determines the Laplacian multiplicities, and compares both maxima with the known metric dimension.
\end{remark}

\section{\rev{Open Problems}}\label{sec:conclusion}
{

The results above leave the following questions.

\begin{problem}
Characterize all connected graphs \(G\) satisfying
\(\dim(G)=\rev{\mm_X}(G),\qquad X\in\{L,Q\}.\)
In particular, identify structural mechanisms, beyond twin classes, that force the maximum spectral multiplicity to equal the metric dimension.
\end{problem}

\begin{problem}
Determine all theta graphs \(\Theta_{p,q,r}\) satisfying
\(\rev{\mm_L}(\Theta_{p,q,r})=\dim(\Theta_{p,q,r})\)
and, separately,
\(\rev{\mm_Q}(\Theta_{p,q,r})=\dim(\Theta_{p,q,r}).\)
\rev{The metric dimension is already known~\cite{Knor2022}; the remaining task here is to determine the maximum spectral multiplicities and compare them with that value.} More generally, determine to what extent the equality behavior is governed by congruence information on the path lengths.
\end{problem}

\begin{problem}
Do there exist infinite families of connected twin-free graphs satisfying exactly one of the two equalities
\(\rev{\mm_L}(G)=\dim(G),\qquad \rev{\mm_Q}(G)=\dim(G)?\)
The computational census shows that finite twin-free examples occur in both directions.
\end{problem}

\begin{problem}
Let \(e_L(n)\) and \(e_Q(n)\) denote the numbers of connected nonisomorphic graphs of order \(n\) satisfying the Laplacian and signless Laplacian equality conditions, respectively. Determine their precise asymptotic growth. In particular, determine whether the limits
\[
\alpha_L=\lim_{n\to\infty}\frac{\log_2 e_L(n)}{n^2},
\qquad
\alpha_Q=\lim_{n\to\infty}\frac{\log_2 e_Q(n)}{n^2}
\]
exist, and, if so, determine their values and whether \(\alpha_L=\alpha_Q\).
Corollary~\ref{CoreLeafExponentialGrowth} places both lower and upper limiting exponents in \([1/18,1/2]\). \rev{The upper endpoint is the general graph-counting bound; narrowing this interval requires information beyond the present core--leaf construction.}
\end{problem}

\begin{problem}
\rev{A subset \(M\subseteq V(G)\) is a \emph{module} if every vertex outside \(M\) is adjacent to either all or none of \(M\); a modular partition is a partition into such subsets. Which equitable partitions, as defined in Section~\ref{sec:twins}, or modular partitions yield lower bounds on metric dimension analogous to \(\dim(G)\geq\rev{\te}(G)\)? When can their invariant-subspace decompositions prove equality with a maximum eigenvalue multiplicity?}
\end{problem}

\begin{problem}
Extend Theorem~\ref{CoreLeafEquality} from pendant leaves to repeated rooted branches. More precisely, let \(H\) be a connected graph with vertices \(v_1,\ldots,v_h\), and \rev{for every \(i\in\{1,\ldots,h\}\) choose \(p_i\in\mathbb{Z}_{\geq2}\) and attach \(p_i\) isomorphic copies of a graph \(F_i\) with a distinguished vertex called its root} at \(v_i\), identifying the root of each copy with \(v_i\). Determine conditions on \(F_1,\ldots,F_h\) under which
\[
\rev{\mm_L}(G)=\dim(G)
\qquad\text{or}\qquad
\rev{\mm_Q}(G)=\dim(G).
\]
In particular, characterize the cases in which
\(\rev{\mm_L}(G)=\rev{\mm_Q}(G)=\dim(G),\)
and enumerate the resulting graphs up to isomorphism. Even the uniform case \(F_1=\cdots=F_h=F\), with a fixed nontrivial rooted graph \(F\), appears to be a natural first step.
\end{problem}
}

\section*{\rev{Acknowledgements}}
\rev{This paper was produced in an AI-assisted workflow (model GPT-6 Astra). The authors are responsible for the content of the paper.}


\begin{thebibliography}{99}

\bibitem{Abreu2012}
N. M. Abreu, D. M. Cardoso, E. A. Martins, M. Robbiano, and B. San Martin,
On the Laplacian and signless Laplacian spectrum of a graph with \(k\) pairwise co-neighbor vertices,
\textit{Linear Algebra Appl.}
437 (2012), 2308--2316.

\bibitem{AbreuJustelMarkenzon2020}
N. Abreu, C. M. Justel, and L. Markenzon,
Strictly chordal graphs: Structural properties and integer Laplacian eigenvalues,
\textit{Linear Algebra Appl.}
682 (2024), 351--362.

\bibitem{AndradeCardosoPastenRojo2015}
E. Andrade, D. M. Cardoso, G. Past\'en, and O. Rojo,
On Faria's inequality for the Laplacian and signless Laplacian
spectra: a unified approach,
\textit{Linear Algebra Appl.}
472 (2015), 81--96.

\bibitem{Babai1995}
L. Babai,
Automorphism groups, isomorphism, reconstruction,
in \textit{Handbook of Combinatorics}, Vol.~2,
R. L. Graham, M. Gr\"otschel, and L. Lov\'asz (eds.),
North-Holland, Amsterdam, 1995, pp.~1447--1540.

\bibitem{Barrett2017}
W. Barrett, A. Francis, and B. Webb,
Equitable decompositions of graphs with symmetries,
\textit{Linear Algebra Appl.}
513 (2017), 409--434.

\bibitem{Brouwer2012}
A. E. Brouwer and W. H. Haemers,
\textit{Spectra of Graphs},
Springer, 2012.

\bibitem{Cardoso2007}
D. M. Cardoso, C. Delorme, and P. Rama,
Laplacian eigenvectors and eigenvalues and almost equitable partitions,
\textit{European J. Combin.}
28 (2007), 665--673.

\bibitem{CardosoHJoin2013}
D. M. Cardoso, M. A. de Freitas, E. A. Martins, and M. Robbiano,
Spectra of graphs obtained by a generalization of the join graph
operation,
\textit{Discrete Math.}
313 (2013), 733--741.

\bibitem{Chartrand2000}
G. Chartrand, L. Eroh, M. A. Johnson, and O. R. Oellermann,
Resolvability in graphs and the metric dimension of a graph,
\textit{Discrete Appl. Math.}
105 (2000), 99--113.

\bibitem{Cvetkovic2007}
D. Cvetkovi\'c, P. Rowlinson, and S. K. Simi\'c,
Signless Laplacians of finite graphs,
\textit{Linear Algebra Appl.}
423 (2007), 155--171.

\bibitem{Cvetkovic2010}
D. M. Cvetkovi\'c, P. Rowlinson, and S. K. Simi\'c,
\textit{An Introduction to the Theory of Graph Spectra},
Cambridge University Press, 2010.

\bibitem{EstradaBenziCoreSatellite}
E. Estrada and M. Benzi,
Core--satellite graphs: Clustering, assortativity and spectral properties,
\textit{Linear Algebra Appl.}
517 (2017), 30--52.

\bibitem{Faria1985}
I. Faria,
Permanental roots and the star degree of a graph,
\textit{Linear Algebra Appl.}
64 (1985), 255--265.

\bibitem{FarhanBaskoro2026}
M. Farhan and E. T. Baskoro,
On the metric dimension and spectrum of graphs,
\textit{Commun. Combin. Optim.}
11 (2026), no.~2, 581--598.

\bibitem{FarhanBaskoroFurther2026}
M. Farhan and E. T. Baskoro,
Further results on the metric dimension and spectrum of graphs,
\textit{CUBO, A Mathematical Journal}
28 (2026), no.~1, 27--42.
\doi{10.56754/0719-0646.2801.027}

\bibitem{GodsilRoyle2001}
C. Godsil and G.F. Royle,
\textit{Algebraic Graph Theory},
Springer, 2001.

\bibitem{GroneMerris1994}
R. Grone and R. Merris,
The Laplacian spectrum of a graph II,
\textit{SIAM J. Discrete Math.}
7 (1994), 221--229.

\bibitem{HararyMelter1976}
F. Harary and R. A. Melter,
On the metric dimension of a graph,
\textit{Ars Combin.}
2 (1976), 191--195.

\bibitem{Hernando2010}
C. Hernando, M. Mora, I. M. Pelayo, C. Seara, and D. R. Wood,
Extremal graph theory for metric dimension and diameter,
\textit{Electron. J. Combin.}
17 (2010), Article R30.
\doi{10.37236/302}

\bibitem{Iswadi2008}
H. Iswadi, E. T. Baskoro, R. Simanjuntak, and A. N. M. Salman,
The metric dimension of graph with pendant edges,
\textit{J. Combin. Math. Combin. Comput.}
65 (2008), 139--145.

\bibitem{Khuller1996}
S. Khuller, B. Raghavachari, and A. Rosenfeld,
Landmarks in graphs,
\textit{Discrete Appl. Math.}
70 (1996), 217--229.

\bibitem{Kirkland2023}
S. Kirkland, H. Monterde, and S. Plosker,
Quantum state transfer between twins in weighted graphs,
\textit{J. Algebr. Combin.}
58 (2023), 623--649.


\bibitem{Knor2022}
M. Knor, J. Sedlar, and R. \v{S}krekovski,
Remarks on the vertex and the edge metric dimension of \(2\)-connected graphs,
\textit{Mathematics} 10 (2022), no.~14, Article 2411.
\doi{10.3390/math10142411}


\bibitem{KooijWindmill}
R. E. Kooij,
On generalized windmill graphs,
\textit{Linear Algebra Appl.}
565 (2019), 25--46.


\bibitem{Okamoto2010}
F. Okamoto, B. Phinezy, and P. Zhang,
The local metric dimension of a graph,
\textit{Math. Bohem.}
135 (2010), no.~3, 239--255.
\doi{10.21136/MB.2010.140702}


\bibitem{PaonePaone2026}
A. Paone and M. Paone,
Line-graph inertia of roses and generalized theta graphs,
preprint, 2026, doi:10.5281/zenodo.21744051.


\bibitem{ReadWilson1998}
R. C. Read and R. J. Wilson,
\textit{An Atlas of Graphs},
Oxford University Press, 1998.


\bibitem{SaravananHJoin2021}
M. Saravanan, S. P. Murugan, and G. Arunkumar,
A generalization of Fiedler's lemma and the spectra of \(H\)-join of
graphs,
\textit{Linear Algebra Appl.}
625 (2021), 20--43.


\bibitem{SimanjuntakAmalgamation2015}
R. Simanjuntak, S. Uttunggadewa, and S. W. Saputro,
Metric dimension for amalgamations of graphs,
in \textit{Combinatorial Algorithms (IWOCA 2014)},
Lecture Notes in Computer Science, vol.~8986,
Springer, 2015, pp.~330--337.
\doi{10.1007/978-3-319-19315-1_29}


\bibitem{Slater1975}
P. J. Slater,
Leaves of trees,
\textit{Congr. Numer.}
14 (1975), 549--559.

\bibitem{Wright1975}
E. M. Wright,
The probability of connectedness of a large unlabelled graph,
\textit{J. London Math. Soc.} (2)
11 (1975), 13--16.

\end{thebibliography}
\end{document}